\documentclass[a4paper,fleqn]{cas-sc}

\usepackage[utf8]{inputenc}
\usepackage[numbers,sort&compress]{natbib}
\usepackage{amsthm,mathtools}
\usepackage{enumitem}
\usepackage{needspace}
\usepackage{microtype}

\ExplSyntaxOn
\keys_set:nn { stm / mktitle } { nologo = true }
\ExplSyntaxOff

\hypersetup{
  pdftitle={Frobenius Profiles of Polynomial-Multiple MRD Codes},
  pdfauthor={Yangcheng Li and Pingzhi Yuan},
  pdfkeywords={MRD codes, Frobenius intersection profiles, polynomial-multiple codes, extension bounds},
  pdfdisplaydoctitle=true,
  pdflang={en-US},
  bookmarksnumbered=true,
  bookmarksopen=true,
  hidelinks
}
\numberwithin{equation}{section}
\allowdisplaybreaks
\newtheorem{theorem}{Theorem}[section]
\newtheorem{lemma}[theorem]{Lemma}
\newtheorem{proposition}[theorem]{Proposition}
\newtheorem{corollary}[theorem]{Corollary}
\theoremstyle{definition}

\newtheorem{example}[theorem]{Example}
\theoremstyle{remark}
\newtheorem{remark}[theorem]{Remark}
\newcommand{\F}{\mathbb F}

\newcommand{\GL}{\operatorname{GL}}
\newcommand{\coeff}{\operatorname{coeff}}
\newcommand{\polylcm}{\operatorname{lcm}}
\newcommand{\Gr}{\operatorname{Gr}}

\newcommand{\qbinom}[2]{\genfrac{[}{]}{0pt}{}{#1}{#2}_{q}}

\begin{document}
\let\WriteBookmarks\relax
\def\floatpagepagefraction{1}
\def\textpagefraction{.001}
\shorttitle{Frobenius Profiles of Polynomial-Multiple MRD Codes}
\shortauthors{Li and Yuan}
\title[mode=title]{Frobenius Profiles of Polynomial-Multiple MRD Codes}

\author[1]{Yangcheng Li}
\cormark[1]
\ead{liyc@m.scnu.edu.cn}
\author[1]{Pingzhi Yuan}
\ead{yuanpz@scnu.edu.cn}
\affiliation[1]{
  organization={School of Mathematical Sciences, South China Normal University},
  addressline={No. 55, West Zhongshan Avenue, Tianhe District},
  city={Guangzhou},
  postcode={510631},
  state={Guangdong},
  country={China}}
\cortext[1]{Corresponding author}

\begin{abstract}
	We study Frobenius intersection profiles of maximum rank
	distance (MRD) codes arising from ordinary polynomial
	multiples over finite fields. We prove that every numerically
	admissible profile is realized
	within this family over every sufficiently large extension,
	using a root-exponent pattern independent of the extension degree.
	A factor-orbit gap formula provides the underlying
	structural description, while effective extension-degree
	bounds guarantee realization. We also establish a lower
	bound valid for all polynomial generators and show that
	the upper and lower bounds differ by only
	$O_q(\log n)$ along an infinite family of parameters.
	Finally, we characterize the generalized Gabidulin
	members and obtain non-generalized-Gabidulin realizations
	of every profile with initial loss at least two over
	infinitely many extensions.
\end{abstract}

\maketitle
\hypersetup{
  pdfauthor={Yangcheng Li and Pingzhi Yuan},
  pdfsubject={Manuscript prepared for submission to Finite Fields and Their Applications}
}

\section{Introduction and main results}
\label{sec:introduction}

Rank-metric codes originated in the pioneering work of Delsarte and Gabidulin, whose fundamental contributions laid the groundwork for the subsequent development of rank-metric coding theory~\cite{Delsarte1978,Gabidulin1985}.
Frobenius intersections provide useful invariants for studying
the structure and equivalence of these codes.
Let $E=\mathbb F_{q^m}$ and let $C\le E^n$ be an $E$-linear
code. Write $C^{q^i}$ for the coordinatewise $q^i$-th
Frobenius image of $C$, and define
\[
d_j(C)=\dim_E\bigcap_{i=0}^j C^{q^i}
\qquad(j\ge0).
\]
The successive losses
$\delta_j=d_{j-1}(C)-d_j(C)$ measure the decrease
in dimension as further Frobenius conjugates are
included in the intersection.
These dimensions are invariant under semilinear
rank equivalence.

Neri, Puchinger and Horlemann-Trautmann
established general properties of these intersection
sequences~\cite{NPH2020}.
In particular, for a maximum rank distance (MRD)
code of dimension $k<n$ and codimension $r=n-k$,
the positive losses form a nonincreasing partition
of $k$ whose largest part is at most $r$.
This raises a natural realization question:
can every sequence satisfying these numerical
conditions occur as the Frobenius intersection
losses of an MRD code?

We answer this question within a concrete family
of polynomial-multiple codes.
For a monic polynomial $g\in E[X]$ of degree
$r<n$, put $k=n-r$ and define
\begin{equation}
	C_g(n)=
	\{\operatorname{coeff}_n(ug):
	u\in E[X],\ \deg u<k\},
	\label{eq:model-intro}
\end{equation}
where $\operatorname{coeff}_n$ records the
$n$ coefficients of a polynomial, padding
with zeros when necessary.
Multiplication is ordinary polynomial multiplication,
without reduction modulo an ambient polynomial.
Equivalently, $C_g(n)$ is the coefficient realization of
the degree-truncated principal ideal $(g)\cap E[X]_{<n}$.
Its generator matrix consists of successive shifts of the
coefficient vector of $g$ and therefore has a Toeplitz-type
structure. This combination of a concrete generator matrix
and explicit divisibility makes the family a natural setting
for Frobenius intersections. We show that it is rich enough
to realize every numerically admissible MRD profile.

Beyond existence, we investigate the extension
degree required for realization.
For fixed $q,n,r$, let
$M_{\mathrm{univ}}(q,n,r)$ be the least integer
$M\ge n$ such that every admissible profile
is realized by an MRD code $C_g(n)$ over
$\mathbb F_{q^m}$ for every integer $m\ge M$.
Here a profile is called \emph{admissible}
if it is determined by positive integers
$\delta_1,\ldots,\delta_h$ satisfying
\begin{equation}
	\begin{gathered}
		r\ge\delta_1\ge\cdots\ge\delta_h\ge1,
		\qquad
		\sum_{j=1}^h\delta_j=k,\\
		d_j=k-\sum_{i=1}^j\delta_i
		\qquad(0\le j\le h),
	\end{gathered}
	\label{eq:admissible-intro}
\end{equation}
with $d_j=0$ for $j>h$.
The quantifiers in the definition of $M_{\mathrm{univ}}$
are ordered as follows: for every $m\ge M$ and every
admissible loss sequence $\boldsymbol\delta$, there is a
monic $g_{\boldsymbol\delta,m}\in\mathbb F_{q^m}[X]$ of
degree $r$ such that $C_{g_{\boldsymbol\delta,m}}(n)$ is MRD
and has profile $\boldsymbol\delta$. Thus the generator may
depend on both the profile and $m$; what is uniform is the
threshold, not the generator.

Our first main result establishes complete
realization and an effective uniform threshold.

\begin{theorem}
	\label{thm:main-realization}
	Fix a prime power $q$ and integers $1\le r<n$,
	and put $k=n-r$.
	Every admissible profile is realized by an
	MRD polynomial-multiple code over every
	sufficiently large extension of $\mathbb F_q$.
	If $r,k\ge2$, then
	\begin{equation}
		M_{\mathrm{univ}}(q,n,r)
		\le
		\max\left\{
		n,\,
		rk+n-\left\lceil\frac{k}{r}\right\rceil
		+\left\lceil
		\log_q\left(\frac{8(n-1)}{q-1}\right)
		\right\rceil
		\right\}.
		\label{eq:main-upper-bound}
	\end{equation}
	If $\min\{r,k\}=1$, then
	$M_{\mathrm{univ}}(q,n,r)=n$.
\end{theorem}

The realization is constructive at the
level of root exponents: for each admissible loss sequence,
the exponent set can be fixed independently of $m$.
Specifically, we construct
\[
T=\{0=e_1<\cdots<e_r=n-h-1\}
\]
and consider
\[
g_{T,\alpha}(X)
=\prod_{i=1}^r(X-\alpha^{q^{e_i}}).
\]
The cyclic gaps between the exponents
determine the required intersection profile.
The MRD property is then established by
generalized Moore determinant tests.
Refining these tests by removing their
ordinary Moore factors yields the stated
extension bound, together with quantitative
estimates for the successful parameters.
The detailed results appear in
Theorem~\ref{thm:universal-profile-realization}
and Corollaries~\ref{cor:polynomial-profile-bound}
and~\ref{cor:boundary-threshold}.

The next question is how close this upper
bound is to the actual threshold.
Our second main result gives a lower bound
that applies to the entire polynomial-multiple
family, rather than merely to the preceding
construction.

\begin{theorem}
	\label{thm:main-lower-bound}
	Fix a prime power $q$ and integers $1\le r<n$.
	If $e$ is an integer satisfying
	$2\le e\le r$ and $n\ge re+2$, then
	\begin{equation}
		M_{\mathrm{univ}}(q,n,r)\ge e(n-1)+1.
		\label{eq:main-lower-bound}
	\end{equation}
	In particular, for every integer $s\ge2$,
	\begin{align}
		s^3+s+1
		\le M_{\mathrm{univ}}(q,s^2+2,s)
		\le s^3+s+2+
		\left\lceil
		\log_q\left(\frac{8(s^2+1)}{q-1}\right)
		\right\rceil.
		\label{eq:main-near-sharp}
	\end{align}
	Thus, for fixed $q$, the upper and lower
	bounds differ by only $O_q(\log n)$
	along this parameter sequence.
\end{theorem}

The lower bound follows from a rigidity property of a suitably chosen Frobenius intersection profile, which allows us to exclude its realization at a specific extension degree. This obstruction applies to the entire polynomial-multiple MRD family and rules out a uniform every-extension threshold of order $O_q(n)$.

\subsection{Methods and relation to earlier work}

The numerical restrictions on Frobenius
intersection sequences are known
\cite{NPH2020}.
Our principal contribution is to prove their
sufficiency within the polynomial family
\eqref{eq:model-intro}, together with effective
bounds on the required extension degree.

The structural tool is a factor-orbit gap
formula for the degrees
\[
D_j(g)=
\deg\operatorname{lcm}
(g,\sigma(g),\ldots,\sigma^j(g)),
\qquad \sigma(a)=a^q.
\]
It expresses these degrees through weighted
gaps in Frobenius orbits of irreducible factors,
including repeated and nonsplit factors.
The intersection dimensions are recovered by
\[
d_j(C_g)=\max\{0,n-D_j(g)\}.
\]
The truncation at zero is important:
the observable profile need not determine
the complete factor-orbit structure.

Our construction uses generalized Moore evaluations on the
short power-basis space
\[
W_\alpha=\langle1,\alpha,\ldots,\alpha^{n-1}
\rangle_{\mathbb F_q}.
\]
The evaluation framework is classical
\cite{Gabidulin1985,KshevetskiyGabidulin2005,NPH2020},
as is the quotient of a generalized Moore determinant by the
ordinary Moore determinant~\cite{BartoliZhou2020}.
The important distinction is the domain of the determinant tests.
The global Moore-exponent-set property requires nonvanishing on
\emph{every} independent tuple in $\mathbb F_{q^m}$; for $q>5$,
Bartoli and Zhou show that a fixed non-arithmetic-progression
exponent set eventually fails that global property as $m$ grows
\cite{BartoliZhou2020}.
Our MRD test requires nonvanishing only on independent tuples
\emph{in $W_\alpha$}, where $n=\dim_{\mathbb F_q}W_\alpha$
remains fixed while $m$ varies. Thus the realization theorem
does not assert, or require, the global Moore-exponent-set
property. The specialized determinant quotients give exact
degree formulas on this restricted space and lead to effective
extension bounds.

Neri et al. established quantitative abundance results
for MRD and Gabidulin codes using subspace-indexed
determinant tests and degree counts~\cite{NeriEtAl2018}.
Here we apply this viewpoint to prescribed Frobenius
root patterns on short power-basis spaces, where specialized
Moore quotients provide exact degree bounds and effective
extension estimates.

Generalized Gabidulin codes
\cite{Gabidulin1985,KshevetskiyGabidulin2005}
provide the natural comparison class; non-Gabidulin MRD
constructions are also well established~\cite{Sheekey2016}.
Using the intersection criterion of
Horlemann-Trautmann and Marshall~\cite{HTM},
we characterize the Gabidulin members by
their Frobenius root chains and prove that
every admissible profile with
$\delta_1\ge2$ has a non-generalized-Gabidulin
realization over infinitely many extensions.
Moreover, writing $p_{\le r}(k)$ for the
number of partitions of $k$ with largest
part at most $r$, complete realization
gives exactly $p_{\le r}(k)$ attainable
profiles and at least that many semilinear
rank-equivalence classes.
This is a lower bound on equivalence classes,
not a complete classification.

The paper is organized as follows.
Section~\ref{sec:models} introduces the
polynomial model and the necessary profile
conditions.
Section~\ref{sec:global-hierarchy} develops
the factor-orbit gap formula.
Sections~\ref{sec:profile-realization}
and~\ref{sec:profile-field-bounds} establish
complete realization and effective extension
bounds.
Section~\ref{sec:profile-lower-bounds}
proves the lower bound, while
Section~\ref{sec:classification} treats
the Gabidulin comparison.

\section{Polynomial codes and Frobenius profiles}
\label{sec:models}

Throughout, let $q$ be a prime power, $E=\mathbb F_{q^m}$, and
\[
1\le k<n\le m,\qquad r=n-k.
\]
Unless otherwise stated, all code dimensions are taken over $E$.
For a vector $c=(c_0,\ldots,c_{n-1})\in E^n$, its rank weight
and the associated rank distance are defined by
\[
\operatorname{wt}_q(c)
=\dim_{\mathbb F_q}
\langle c_0,\ldots,c_{n-1}\rangle_{\mathbb F_q},
\qquad
d_q(c,c')=\operatorname{wt}_q(c-c').
\]
The minimum rank distance of a nonzero linear code
$C\le E^n$ is
\[
d_q(C)=\min_{0\ne c\in C}\operatorname{wt}_q(c).
\]
Since rank weight does not exceed Hamming weight, the
classical Singleton bound implies
\[
d_q(C)\le n-\dim_E C+1.
\]
A code attaining equality is called a
\emph{maximum rank distance} (MRD) code.
In particular, since $\dim_E C_g(n)=k$, the polynomial code
$C_g(n)$ is MRD precisely when $d_q(C_g(n))=r+1$.

For $f(X)=\sum_{i=0}^{n-1}f_iX^i\in E[X]_{<n}$, write
\[
\operatorname{coeff}_n(f)=(f_0,\ldots,f_{n-1})\in E^n,
\]
where missing coefficients are understood to be zero.
If $g\in E[X]$ is monic of degree $r$, then
\[
\operatorname{coeff}_n(g),\,
\operatorname{coeff}_n(Xg),\,\ldots,\,
\operatorname{coeff}_n(X^{k-1}g)
\]
form an $E$-basis of $C_g(n)$.
Indeed, multiplication by the nonzero polynomial $g$ is
injective on $E[X]_{<k}$, and every resulting product has
degree less than $n$. Equivalently,
\[
C_g(n)
=\operatorname{coeff}_n\bigl((g)\cap E[X]_{<n}\bigr),
\]
where $(g)$ denotes the principal ideal generated by $g$
in $E[X]$. Thus the construction uses ordinary polynomial
multiplication with a degree restriction, without reduction
modulo an ambient polynomial.

Finally, for $s\ge0$, let
\[
\sigma_s(a)=a^{q^s}\qquad(a\in E),
\qquad \sigma=\sigma_1.
\]
These Frobenius maps act coordinatewise on $E^n$ and
coefficientwise on $E[X]$; in particular, the indeterminate
$X$ is fixed.

\subsection{Intersections and the MRD test}

\begin{lemma}
\label{lem:intersection}
If \(g,h\) are monic of degree \(r<n\), then
\begin{equation}
 \dim_E(C_g(n)\cap C_h(n))
       =\max\{0,n-2r+\deg\gcd(g,h)\}.
 \label{eq:gcd-intersection}
\end{equation}
Moreover, \(C_g(n)=C_h(n)\) implies \(g=h\).
\end{lemma}

\begin{proof}
Let $L=\polylcm(g,h)$. A polynomial of degree less than $n$
belongs to both codes precisely when it is divisible by $L$, so
\[
C_g(n)\cap C_h(n)
=\coeff_n\bigl((L)\cap E[X]_{<n}\bigr).
\]
This space has dimension $\max\{0,n-\deg L\}$.
Since $\deg L=2r-\deg\gcd(g,h)$, this proves
\eqref{eq:gcd-intersection}.
If $C_g(n)=C_h(n)$, then $\coeff_n(g)\in C_h(n)$, hence $h\mid g$.
Both polynomials are monic of degree $r$, so $g=h$.
\end{proof}

We recall the standard parity-check criterion for MRD codes
\cite[Proposition~2.11]{HTM}.

\begin{proposition}\label{prop:mrd-criterion}
	Let $H\in E^{r\times n}$ have rank $r$, and put $C=\ker H$.
	Then $C$ is MRD if and only if
	\[
	\det(HA)\ne0
	\qquad
	\text{for every }A\in\F_q^{n\times r}
	\text{ of rank }r.
	\]
	It suffices to choose one basis matrix for each
	$r$-dimensional subspace of $\F_q^n$.
\end{proposition}

\subsection{Admissible loss sequences}

For an $E=\mathbb F_{q^m}$-linear code $C\le E^n$ and a generator
$\sigma\in\operatorname{Gal}(E/\mathbb F_q)$, write
\[
 I_j(C,\sigma)=\bigcap_{t=0}^j\sigma^tC,\qquad
 d_j(C,\sigma)=\dim_E I_j(C,\sigma)\qquad(j\ge0).
\]
All Frobenius maps act coordinatewise.

The following proposition specializes the general properties
of Frobenius intersection sequences
\cite[Proposition~3.6]{NPH2020} to MRD codes.
We include a proof for completeness, since these numerical
conditions form the starting point of our realization theorem.

\begin{proposition}
	\label{prop:necessary-profile}
	Suppose $1\le k<n\le m$ and $C$ is an $E$-linear MRD code
	of dimension $k$, with $r=n-k$.
	Then there is an integer $h$, $1\le h\le k$, such that
	\[
	k=d_0>d_1>\cdots>d_{h-1}>d_h=0,
	\qquad d_j=0\quad(j\ge h).
	\]
	Its positive losses $\delta_j=d_{j-1}-d_j$ satisfy
	\begin{equation}\label{eq:admissible-losses}
		r\ge\delta_1\ge\delta_2\ge\cdots\ge\delta_h\ge1,
		\qquad \sum_{j=1}^h\delta_j=k.
	\end{equation}
\end{proposition}

\begin{proof}
	We have $I_{j+1}=I_j\cap\sigma I_j$.
	For $j\ge1$, consider the map
	\[
	I_j/I_{j+1}\longrightarrow I_{j-1}/I_j,
	\qquad
	x+I_{j+1}\longmapsto\sigma^{-1}x+I_j.
	\]
	This map is well-defined, injective, and
	$\sigma^{-1}$-semilinear.
	Indeed, $\sigma^{-1}I_j\subseteq I_{j-1}$, and
	for $x\in I_j$, the condition $\sigma^{-1}x\in I_j$
	is equivalent to
	$x\in I_j\cap\sigma I_j=I_{j+1}$.
	Comparing the dimensions of the two quotient spaces gives
	\[
	d_j-d_{j+1}\le d_{j-1}-d_j.
	\]
	Thus the successive losses are nonincreasing.
	
	Next, suppose that $I_j=I_{j+1}\ne0$ for some $j\ge0$.
	Since $I_{j+1}=I_j\cap\sigma I_j$ and
	$\dim_E(\sigma I_j)=\dim_E I_j$, we obtain
	$I_j=\sigma I_j$.
	Let $R$ be the unique reduced-row-echelon generator matrix
	of $I_j$.
	Since $\sigma I_j=I_j$, applying $\sigma$ entrywise to $R$
	gives another reduced-row-echelon matrix with the same
	row space. By uniqueness, $\sigma R=R$.
	As the fixed field of $\sigma$ is $\F_q$, the nonzero
	rows of $R$ belong to $\F_q^n$.
	Consequently, $I_j\subseteq C$ contains a nonzero
	rank-one codeword, contradicting the MRD minimum distance
	$r+1\ge2$.
	Hence the intersection dimensions decrease strictly
	until they reach zero. Since $d_0=k$, this occurs
	at some index $h\le k$.
	
	Finally, the dimension inequality
	\[
	d_1=\dim_E(C\cap\sigma C)
	\ge 2k-n=k-r
	\]
	implies $\delta_1=k-d_1\le r$.
	Since $d_0=k$ and $d_h=0$, telescoping gives
	\[
	\sum_{j=1}^h\delta_j
	=\sum_{j=1}^h(d_{j-1}-d_j)=k.
	\]
	This completes the proof.
\end{proof}

\subsection{Semilinear rank equivalence}
\label{sec:rank-equivalence}

Two $E$-linear codes $C,D\le E^n$ are \emph{semilinearly rank equivalent}
if $D=\lambda\tau(C)B$, where $\lambda\in E^\times$,
$\tau\in\operatorname{Aut}(E)$ acts coordinatewise, and
$B\in\GL_n(q)$ \cite[Theorem~2.5]{NPH2020}.
Since $C$ is $E$-linear, the scalar $\lambda$ may be omitted.
Finite-field automorphisms commute, and $\sigma(B)=B$ for every
$\sigma\in\operatorname{Gal}(E/\mathbb F_q)$. Hence
\[
\bigcap_{i=0}^j\sigma^i(D)
=\tau\left(\bigcap_{i=0}^j\sigma^i(C)\right)B
\qquad(j\ge0).
\]
Thus each Frobenius intersection dimension is invariant under
semilinear rank equivalence \cite[Lemma~3.1]{NPH2020}.
We use this invariant in Section~\ref{sec:classification} to distinguish
codes with different profiles.

\section{The factor-orbit gap formula}
\label{sec:global-hierarchy}

In this section, $\sigma$ may be any generator of
$\operatorname{Gal}(E/\mathbb F_q)$; we suppress it from
$d_j(C,\sigma)$ when no confusion arises.

\subsection{Factor orbits and the gap formula}

For $j\ge0$, define the monic least common multiple
\[
L_j(g)=\operatorname{lcm}
\bigl(g,\sigma(g),\ldots,\sigma^j(g)\bigr),
\qquad D_j(g)=\deg L_j(g).
\]
Since $\sigma^i$ preserves polynomial multiplication and is
a bijection on $E[X]_{<k}$, we have
$\sigma^i(C_g(n))=C_{\sigma^i(g)}(n)$.
A polynomial belongs to all these codes precisely when
it is divisible by $L_j(g)$. Consequently,
\begin{equation}\label{eq:global-lcm}
	\begin{aligned}
		\bigcap_{i=0}^{j}\sigma^i(C_g(n))
		&=\bigl\{\coeff_n(L_j(g)u):
		u\in E[X],\ \deg(L_j(g)u)<n\bigr\},\\
		d_j(C_g)&=\max\{0,n-D_j(g)\}.
	\end{aligned}
\end{equation}
In particular, the intersection vanishes once $D_j(g)\ge n$.

The basic gap mechanism is already visible when $g$ is
squarefree and all its roots lie in one Frobenius orbit.
If $g(X)=\prod_{i\in T}(X-\sigma^i(\alpha))$, where
$\alpha$ has $\sigma$-orbit length $t$ and
$T\subseteq\mathbb Z/t\mathbb Z$, then
$D_j(g)=|T+\{0,\ldots,j\}|$.
A block of $\ell$ missing orbit positions contributes
$\min\{j,\ell\}$ new roots after the first $j$ shifts.
For general $g$, the same count is made at each
multiplicity level of each irreducible-factor orbit,
with contributions weighted by factor degrees.

To compute $D_j(g)$, factor $g$ into monic irreducible
polynomials over $E$ and retain each distinct $\sigma$-orbit
containing a factor of $g$ once. The full orbits are needed
because a Frobenius conjugate of $g$ may contain irreducible
factors that do not occur in $g$ itself. These orbits are
finite, since $\sigma$ has finite order.
For such an orbit, choose
$P_0\mid g$ and write
\[
\mathcal O=(P_0,\ldots,P_{t-1}),
\qquad P_i=\sigma^i(P_0)
\quad(i\in\mathbb Z/t\mathbb Z).
\]
These are distinct irreducible polynomials of a common
degree, with $\sigma(P_{t-1})=P_0$.
We record the multiplicity at every position, including
those where the factor is absent from $g$. Set
\[
e_{\mathcal O}=\deg P_i,
\qquad a_i=v_{P_i}(g),
\]
where $a_i=0$ when $P_i\nmid g$. Thus
\[
g=\prod_{\mathcal O}\prod_{i=0}^{t-1}P_i^{a_i},
\qquad
r=\sum_{\mathcal O}e_{\mathcal O}
\sum_{i=0}^{t-1}a_i.
\]
This description allows repeated and nonsplit factors,
as well as conjugate factors absent from $g$.
Applying $\sigma$ shifts the multiplicities one position
around each orbit. Thus the same orbit data describe the
factorizations of all the polynomials entering $L_j(g)$.

For each orbit and each multiplicity level
$1\le u\le\max_i a_i$, define the \emph{superlevel set}
\[
S_{\mathcal O,u}
=\{i\in\mathbb Z/t\mathbb Z:a_i\ge u\}.
\]
This is the set of positions where the factor occurs at
least $u$ times. A position of multiplicity $a_i$ belongs
to precisely $a_i$ such sets, so summing their cardinalities
recovers the total multiplicity in that orbit.

For a nonempty proper subset $S$ of $\mathbb Z/t\mathbb Z$,
a \emph{cyclic gap} is a maximal nonempty string of
consecutive indices outside $S$. Its length is the number
of missing positions between two successive elements of
$S$. The full orbit has no gaps.

For $\ell\ge1$, let $\nu_\ell$ count all gaps of length
$\ell$ in all sets $S_{\mathcal O,u}$, assigning each gap
the weight $e_{\mathcal O}$. Distinct gaps of the same
length are counted separately; a gap occurring at several
multiplicity levels contributes once at each level.
Full superlevel sets contribute zero.
For example, if $S_{\mathcal O,u}$ has $b$ gaps of length
$\ell$, it contributes $b e_{\mathcal O}$ to $\nu_\ell$.
The weighting reflects the degree contributed by each
newly occupied factor position; counting the levels separately
reflects its multiplicity.

\begin{theorem}
	\label{thm:orbit-gap-formula}
	For every $j\ge0$,
	\begin{equation}
		D_j(g)=r+\sum_{\ell\ge1}\nu_\ell\min\{j,\ell\}.
		\label{eq:gap-hierarchy}
	\end{equation}
	Consequently,
	\begin{align}
		D_j-D_{j-1}
		&=\sum_{\ell\ge j}\nu_\ell
		\qquad(j\ge1),
		\label{eq:gap-increments}\\
		\nu_j
		&=2D_j-D_{j-1}-D_{j+1}
		\qquad(j\ge1).
		\label{eq:gap-inversion}
	\end{align}
	Thus the full degree hierarchy determines the
	weighted gap-length distribution, but not necessarily
	the positions of the gaps, their factor orbits,
	or the nesting of the superlevel sets.
\end{theorem}

\begin{proof}
	Fix an orbit $\mathcal O=(P_0,\ldots,P_{t-1})$.
	The multiplicity of $P_i$ in $\sigma^s(g)$ is
	$a_{i-s}$, with indices modulo $t$. Hence
	\[
	b_i(j):=v_{P_i}(L_j(g))
	=\max_{0\le s\le j}a_{i-s}.
	\]
	At multiplicity level $u$, this gives
	\[
	b_i(j)\ge u
	\iff i\in S_{\mathcal O,u}+\{0,\ldots,j\},
	\]
	where addition takes place in $\mathbb Z/t\mathbb Z$.
	Indeed, $b_i(j)\ge u$ means that $a_{i-s}\ge u$
	for at least one $s\in\{0,\ldots,j\}$. Thus the
	$u$-th superlevel set of $L_j(g)$ is the union of the
	first $j+1$ shifts of $S_{\mathcal O,u}$.

	For a nonempty proper subset $S$ of this cyclic orbit,
	enlargement to $S+\{0,\ldots,j\}$ fills the first
	$\min\{j,\ell\}$ positions of each gap of length $\ell$,
	starting from its preceding occupied position.
	More explicitly, a gap
	$\{a+1,\ldots,a+\ell\}$ with endpoints
	$a,a+\ell+1\in S$ retains exactly
	$\max\{\ell-j,0\}$ missing positions.
	Since the gaps partition the complement of $S$,
	\[
	|S+\{0,\ldots,j\}|
	=|S|+\sum_{\text{gaps of }S}\min\{j,\ell\}.
	\]
	The identity also holds for the full orbit, with an
	empty gap sum. Summing over all multiplicity levels
	and weighting by the factor degrees yields
	\[
	\begin{aligned}
		D_j(g)
		&=\sum_{\mathcal O}e_{\mathcal O}
		\sum_{u=1}^{\max_i a_i}
		\bigl|S_{\mathcal O,u}+\{0,\ldots,j\}\bigr|\\
		&=r+\sum_{\ell\ge1}\nu_\ell\min\{j,\ell\}.
	\end{aligned}
	\]
	Here the initial contribution is
	$\sum_{\mathcal O}e_{\mathcal O}
	\sum_u|S_{\mathcal O,u}|=\deg g=r$;
	hence full superlevel sets are also accounted for.
	Taking first and second differences proves
	\eqref{eq:gap-increments} and \eqref{eq:gap-inversion}.
\end{proof}

The intersection profile does not necessarily determine
the full degree hierarchy. Indeed, if $h$ is the first
index with $d_h(C_g)=0$, then
\[
D_j(g)=n-d_j(C_g)\quad(0\le j<h),
\]
while for $j\ge h$ only $D_j(g)\ge n$ is known.
Consequently, the profile determines $\nu_j$ through
\eqref{eq:gap-inversion} for $1\le j\le h-2$,
but not necessarily beyond this range.
Even the complete degree hierarchy associated with a
polynomial generator does not necessarily determine the
semilinear rank-equivalence class within the polynomial-multiple
family; see Example~\ref{ex:same-profile-inequivalent}.

\begin{example}
	\label{ex:gap-count}
	Let $n=9$, $m\ge9$, and choose
	$\alpha\in\mathbb F_{q^m}$ of degree $m$ over $\mathbb F_q$.
	Consider the root-exponent set $T=\{0,2,5\}$ and
	\[
	g(X)=(X-\alpha)(X-\alpha^{q^2})(X-\alpha^{q^5}).
	\]
	The three linear factors belong to a single Frobenius
	orbit of length $m$. Since $g$ is squarefree,
	there is only one nonempty multiplicity level,
	whose superlevel set is $S=T\subseteq\mathbb Z/m\mathbb Z$.
	
	The three cyclic gaps of $S$ are
	\[
	\{1\},\qquad
	\{3,4\},\qquad
	\{6,7,\ldots,m-1\},
	\]
	with lengths $1$, $2$, and $m-6$, respectively.
	Since the factors are linear, each gap has weight one.
	Thus
	\[
	\nu_1=\nu_2=\nu_{m-6}=1,
	\]
	and all other weighted gap counts vanish.
	
	By Theorem~\ref{thm:orbit-gap-formula},
	\[
	D_j(g)
	=3+\min\{j,1\}+\min\{j,2\}
	+\min\{j,m-6\}.
	\]
	In particular, for $0\le j\le3$,
	\[
	(D_0,D_1,D_2,D_3)=(3,6,8,9).
	\]
	Since $n=9$, formula~\eqref{eq:global-lcm} gives
	\[
	(d_0,d_1,d_2,d_3)=(6,3,1,0),
	\]
	with positive losses $(3,2,1)$.
	
	This calculation determines the Frobenius intersection
	profile but does not by itself establish the MRD property.
	In Example~\ref{ex:profile-construction}, we return to
	the same exponent pattern and give sufficient conditions
	ensuring that the resulting code is MRD.
\end{example}

\subsection{Multichains and base-field obstructions}

The multiplicity superlevel sets also provide a canonical
decomposition into consecutive factor chains.

\begin{corollary}
	\label{cor:canonical-multichains}
	Let $g_{\mathrm{fix}}$ be the greatest monic divisor
	of $g$ belonging to $\mathbb F_q[X]$.
	Then $g$ admits a factorization
	\[
	g=g_{\mathrm{fix}}
	\prod_{\lambda}
	\left(
	P_\lambda\,\sigma(P_\lambda)\cdots
	\sigma^{b_\lambda-1}(P_\lambda)
	\right),
	\]
	where each $P_\lambda$ is monic and irreducible
	over $E$, and each displayed chain is shorter
	than the full $\sigma$-orbit of $P_\lambda$.

	The chains are determined by the maximal consecutive
	components of the multiplicity superlevel sets,
	after removing the Frobenius-fixed part.
	This gives a canonical multiset of chains,
	unique up to reordering, and
	\[
	\sum_\lambda\deg P_\lambda
	=
	r-\deg\gcd(g,\sigma(g)).
	\]
	Moreover, if $C_g(n)$ is MRD, then
	$g_{\mathrm{fix}}=1$.
\end{corollary}

\begin{proof}
	For each orbit $\mathcal O=(P_0,\ldots,P_{t-1})$,
	put $c_{\mathcal O}=\min_i a_i$. Since $\sigma$
	cyclically permutes the factors in each orbit,
	the product
	\[
	g_{\mathrm{fix}}
	=\prod_{\mathcal O}
	\left(\prod_{i=0}^{t-1}P_i\right)^{c_{\mathcal O}}
	\]
	is fixed by $\sigma$ and hence belongs to
	$\mathbb F_q[X]$.
	Conversely, every monic base-field divisor of
	$g$ is $\sigma$-invariant, so its multiplicity
	is constant along each orbit and cannot exceed
	$c_{\mathcal O}$. Thus $g_{\mathrm{fix}}$ is
	the greatest monic base-field divisor of $g$.
	
	Now subtract $c_{\mathcal O}$ from each
	multiplicity in $\mathcal O$. Since at least
	one resulting multiplicity is zero, every
	remaining nonempty superlevel set is a proper
	subset of the cyclic orbit. Decompose each
	such set into its maximal consecutive components.
	A component starting at $P_\lambda$ and
	containing $b_\lambda$ positions gives the chain
	\[
	P_\lambda\sigma(P_\lambda)\cdots
	\sigma^{b_\lambda-1}(P_\lambda).
	\]
	Each remaining factor occurrence is accounted
	for at exactly one component at each of its
	multiplicity levels. Hence the product of all
	these chains is $g/g_{\mathrm{fix}}$.
	Since the superlevel sets and their maximal
	components are uniquely determined, this gives
	the asserted canonical decomposition, unique
	up to reordering.
	
	Every nonempty proper subset of a cyclic orbit
	has as many maximal occupied components as
	complementary gaps. Each component contributes
	one chain whose initial irreducible factor has
	degree $e_{\mathcal O}$. Thus, weighting both
	counts by this degree and summing over all
	orbits and multiplicity levels, we obtain
	\[
	\sum_\lambda\deg P_\lambda
	=\sum_{\ell\ge1}\nu_\ell
	=D_1(g)-r
	=r-\deg\gcd(g,\sigma(g)),
	\]
	where the last two equalities follow from
	\eqref{eq:gap-increments} and
	$D_1(g)=2r-\deg\gcd(g,\sigma(g))$.
	
	Finally, suppose that $C_g(n)$ is MRD.
	If $\deg g_{\mathrm{fix}}>0$, write
	$g=g_{\mathrm{fix}}v$, where
	\[
	g_{\mathrm{fix}}=\sum_j h_jX^j,
	\qquad
	v=\sum_{i=0}^{\deg v}v_iX^i,
	\qquad h_j\in\F_q.
	\]
	For every $s$, the coefficient
	\[
	[X^s]g=\sum_{i+j=s}h_jv_i
	\]
	belongs to
	$\langle v_0,\ldots,v_{\deg v}\rangle_{\F_q}$.
	Consequently, the nonzero codeword
	$\coeff_n(g)$ satisfies
	\[
	\operatorname{wt}_q(\coeff_n(g))
	\le\deg v+1
	=r-\deg g_{\mathrm{fix}}+1
	\le r.
	\]
	This contradicts the MRD minimum distance
	$r+1$.
\end{proof}

Completing factors of $g$ to a polynomial over the base
field gives a further obstruction to the MRD property.

\begin{proposition}
	\label{prop:orbit-completion-obstruction}
	Let $F\in\mathbb F_q[X]$ be monic, and put
	$a=\deg\gcd(g,F)$. If $a>0$ and
	$\deg F-a<k$, then
	\[
	d_q(C_g(n))\le r-a+1.
	\]
	Consequently, if $C_g(n)$ is MRD, every monic
	irreducible polynomial $P\in\mathbb F_q[X]$
	satisfying $\gcd(g,P)\ne1$ must have
	\begin{equation}
		\deg P-\deg\gcd(g,P)\ge k.
		\label{eq:orbit-completion}
	\end{equation}
	
	In terms of factor orbits, this condition has
	the following equivalent form.
	Let $\mathcal O=(P_0,\ldots,P_{t-1})$ be any
	$\sigma$-orbit of monic irreducible factors
	over $E$ containing at least one factor of $g$.
	With $e_{\mathcal O}=\deg P_i$ and
	$a_i=v_{P_i}(g)$, we have
	\[
	e_{\mathcal O}
	\#\{i\in\mathbb Z/t\mathbb Z:a_i=0\}\ge k.
	\]
	Thus every occupied orbit must contain missing
	factors whose total degree is at least $k$,
	counting distinct factors without multiplicities.
\end{proposition}

\begin{proof}
	Let $h=\gcd(g,F)$. The polynomial
	\[
	f=\frac{g}{h}F=g\frac{F}{h}
	\]
	is a nonzero multiple of $g$ of degree
	$r+\deg F-a<n$, so $\coeff_n(f)\in C_g(n)$.
	As $F\in\mathbb F_q[X]$, every coefficient of $f$ lies
	in the $\mathbb F_q$-span of the coefficients of $g/h$.
	Therefore
	\[
	\operatorname{wt}_q(\coeff_n(f))
	\le\deg(g/h)+1=r-a+1.
	\]
	For an MRD code, this forces
	$\deg F-\deg\gcd(g,F)\ge k$ whenever the gcd has
	positive degree. Taking $F=P$ proves
	\eqref{eq:orbit-completion}.

	For an occupied factor orbit, the product
	\[
	P_{\mathcal O}=\prod_{i=0}^{t-1}P_i
	\]
	is fixed by $\sigma$ and lies in $\mathbb F_q[X]$.
	It is irreducible over $\mathbb F_q$: a nonconstant
	proper base-field factor would determine a nonempty
	proper $\sigma$-invariant subset of
	$\{P_0,\ldots,P_{t-1}\}$, contrary to transitivity.
	Since $P_{\mathcal O}$ is squarefree over $E$,
	\[
	\deg P_{\mathcal O}=t e_{\mathcal O},
	\qquad
	\deg\gcd(g,P_{\mathcal O})
	=e_{\mathcal O}\#\{i:a_i>0\}.
	\]
	Substitution in \eqref{eq:orbit-completion} gives
	the stated missing-factor inequality.
	Conversely, if an irreducible $P\in\mathbb F_q[X]$
	meets $g$, the orbit product of an irreducible
	$E$-factor of $\gcd(g,P)$ belongs to $\mathbb F_q[X]$
	and divides $P$, so it equals $P$. This proves the
	equivalence of the two formulations.
\end{proof}

\begin{corollary}
	\label{cor:no-base-factor-before-vanishing}
	Suppose that $C_g(n)$ is MRD. If $D_j(g)<n$,
	then $L_j(g)$ has no nonconstant factor in
	$\mathbb F_q[X]$.
\end{corollary}

\begin{proof}
	Suppose, to the contrary, that a monic irreducible
	polynomial $P\in\mathbb F_q[X]$ divides $L_j(g)$.
	We first show that
	\[
	a:=\deg\gcd(g,P)>0.
	\]
	Choose an irreducible factor $Q$ of $P$ over $E$.
	Since $P\mid L_j(g)$, we have $Q\mid L_j(g)$,
	so $Q$ divides $\sigma^i(g)$ for some $0\le i\le j$.
	Applying $\sigma^{-i}$, and using $\sigma(P)=P$,
	shows that $\sigma^{-i}(Q)$ divides both $g$ and $P$.
	Hence $a>0$.
	
	Now both $g$ and $P$ divide $L_j(g)$, so
	\[
	D_j(g)\ge\deg\operatorname{lcm}(g,P)
	=r+\deg P-a.
	\]
	By \eqref{eq:orbit-completion}, the MRD property gives
	$\deg P-a\ge k$. Consequently,
	\[
	D_j(g)\ge r+k=n,
	\]
	contrary to the assumption.
\end{proof}

\section{Complete realization of admissible profiles}
\label{sec:profile-realization}

We now realize every sequence satisfying \eqref{eq:admissible-losses}.
The construction first prescribes the Frobenius root pattern; a
degree argument then ensures the MRD property.

\begin{lemma}
\label{lem:fixed-pattern-mrd}
Fix $q$, integers $1\le r<n$, and nonnegative integers
$e_1<\cdots<e_r$.  Put $k=n-r$ and
\[
 B_T=\sum_{i=1}^r(k+i-1)q^{e_i}.
\]
For every integer $m>B_T$ and every
$\alpha\in\mathbb F_{q^m}$ of degree $m$ over $\mathbb F_q$, the
ordinary polynomial
\[
 g_T(X)=\prod_{i=1}^r(X-\alpha^{q^{e_i}})
\]
defines an MRD code
\[
 C_{g_T}(n)=\{\coeff_n(u g_T):\deg u<k\}.
\]
\end{lemma}

\begin{proof}
	Since
	\[
	m>B_T\ge(n-1)q^{e_r}
	\ge\max\{n-1,e_r\},
	\]
	we have $m\ge n$ and $e_r<m$. As $\alpha$
	has degree $m$ over $\F_q$, the roots
	$\alpha^{q^{e_1}},\ldots,\alpha^{q^{e_r}}$
	of $g_T$ are distinct. A polynomial of degree
	less than $n$ is divisible by $g_T$ precisely
	when it vanishes at all these roots. Thus
	evaluation gives a parity-check matrix
	\[
	H=(\alpha^{j q^{e_i}})_{1\le i\le r,\ 0\le j<n},
	\]
	whose first $r$ columns form a nonsingular
	Vandermonde matrix. By
	Proposition~\ref{prop:mrd-criterion},
	it remains to prove $\det(HA)\ne0$ for every
	rank-$r$ matrix $A\in\F_q^{n\times r}$.
	
	Associate to the columns of $A=(a_{ij})$
	the linearly independent polynomials
	\[
	P_j(Z)=\sum_{i=0}^{n-1}a_{ij}Z^i
	\qquad (1\le j\le r),
	\]
	and put
	\[
	F(Z)=\det(P_j(Z)^{q^{e_i}})_{i,j=1}^r.
	\]
	Since the coefficients of the $P_j$ lie
	in $\F_q$, we have
	\[
	(HA)_{ij}=P_j(\alpha)^{q^{e_i}},
	\]
	and hence $\det(HA)=F(\alpha)$.
	
	An invertible $\F_q$-change of basis
	multiplies $F$ by a nonzero scalar.
	We may therefore choose an echelon basis
	with leading degrees
	\[
	0\le d_1<\cdots<d_r<n,
	\qquad d_i=\deg P_i.
	\]
	There must be $r-i$ distinct pivot degrees
	greater than $d_i$ but smaller than $n$.
	Consequently,
	\[
	d_i+(r-i)\le n-1,
	\qquad\text{so}\qquad d_i\le k+i-1.
	\]
	
	In the determinant expansion, the term
	indexed by $\pi\in S_r$ has degree
	\[
	\sum_{i=1}^r d_{\pi(i)}q^{e_i}.
	\]
	Both $(d_i)$ and $(q^{e_i})$ are strictly
	increasing, so the strict rearrangement
	inequality makes the identity permutation
	the unique maximizer. If $c_i\ne0$ is the
	leading coefficient of $P_i$, this term
	has leading coefficient
	$\prod_i c_i^{q^{e_i}}\ne0$, while every
	other permutation has strictly smaller
	degree. Therefore no cancellation of
	the highest-degree term is possible, and
	\[
	F\ne0,\qquad
	\deg F=\sum_{i=1}^r d_iq^{e_i}
	\le\sum_{i=1}^r(k+i-1)q^{e_i}
	=B_T<m.
	\]
	Since the minimal polynomial of $\alpha$
	over $\F_q$ has degree $m$, the nonzero
	polynomial $F\in\F_q[Z]$ cannot vanish
	at $\alpha$. Hence $\det(HA)=F(\alpha)\ne0$
	for every rank-$r$ matrix $A$, completing
	the MRD test.
\end{proof}

\begin{remark}
	The parity-check matrix in the preceding proof
	also gives a linearized-polynomial evaluation
	description of the dual code.
	Indeed, its rows are the evaluations of
	$Z^{q^{e_1}},\ldots,Z^{q^{e_r}}$ at
	$1,\alpha,\ldots,\alpha^{n-1}$. Hence
	\[
	C_{g_T}(n)^\perp
	=
	\left\{
	\bigl(f(1),f(\alpha),\ldots,f(\alpha^{n-1})\bigr):
	f(Z)=\sum_{i=1}^r b_i Z^{q^{e_i}},
	\quad b_i\in E
	\right\},
	\]
	where $\perp$ denotes the usual Euclidean dual.
	
	This is a generalized Moore evaluation model
	on the $\F_q$-subspace
	\[
	W=\langle 1,\alpha,\ldots,\alpha^{n-1}\rangle_{\F_q}
	\subseteq E,
	\]
	which has dimension $n$ because $\alpha$ has
	degree at least $n$ over $\F_q$.
	The MRD argument establishes the required
	determinant nonvanishing for every
	$\F_q$-independent $r$-tuple in $W$.
	It does not assert the same property for
	arbitrary $\F_q$-independent $r$-tuples in $E$.
	See \cite{NPH2020} for the general evaluation framework.
\end{remark}

\begin{theorem}
	\label{thm:universal-profile-realization}
	Fix a prime power $q$ and integers $1\le k<n$, and
	put $r=n-k$. Let
	\[
	k=d_0>d_1>\cdots>d_{h-1}>d_h=0
	\]
	be any integer sequence whose losses
	$\delta_j=d_{j-1}-d_j$ satisfy
	\eqref{eq:admissible-losses}.
	
	Then there exists a set of nonnegative integers
	\[
	T=\{0=e_1<e_2<\cdots<e_r\},
	\qquad e_r=n-h-1,
	\]
	depending only on the prescribed sequence, with
	the following property: for every integer $m>B_T$
	and every $\alpha\in\mathbb F_{q^m}$ of degree $m$
	over $\mathbb F_q$, the polynomial $g_T$ defined
	in Lemma~\ref{lem:fixed-pattern-mrd} generates
	an MRD code $C_{g_T}(n)$ satisfying
	\[
	d_j(C_{g_T}(n),\sigma)=
	\begin{cases}
		d_j, & 0\le j\le h,\\
		0,   & j>h,
	\end{cases}
	\qquad \sigma(x)=x^q.
	\]
	Here $B_T$ is the degree bound from
	Lemma~\ref{lem:fixed-pattern-mrd}.
	When $r=1$, the set is simply $T=\{0\}$.
	
	Consequently, the conditions in
	Proposition~\ref{prop:necessary-profile}
	are both necessary and sufficient for the
	realization of Frobenius intersection profiles
	of MRD polynomial-multiple codes over
	sufficiently large extensions.
\end{theorem}

\begin{proof}
	We construct a set of $r$ occupied positions
	on the nonnegative integer line. The internal
	gaps, together with one external gap of length
	at least $h$, will give exactly $\delta_j$
	gaps of length at least $j$ for $1\le j\le h$.
	
	\emph{Step 1: Prescribing the internal gaps.}
	We reserve one contribution to each gap count
	for the external gap. For each $1\le j<h$,
	take $\delta_j-\delta_{j+1}$ internal gaps
	of length $j$, and take an additional
	$\delta_h-1$ internal gaps of length $h$.
	These numbers are nonnegative by
	\eqref{eq:admissible-losses}.
	
	The total number of internal gaps is
	\[
	\sum_{j=1}^{h-1}
	(\delta_j-\delta_{j+1})+(\delta_h-1)
	=\delta_1-1,
	\]
	and their total length is
	\begin{equation}\label{eq:internal-gap-total}
		\begin{aligned}
			&\sum_{j=1}^{h-1}
			j(\delta_j-\delta_{j+1})
			+h(\delta_h-1)\\
			&\qquad=\sum_{j=1}^h\delta_j-h
			=k-h.
		\end{aligned}
	\end{equation}
	
	\emph{Step 2: Choosing the occupied runs.}
	The $\delta_1-1$ internal gaps separate
	$\delta_1$ occupied runs. Since
	$\delta_1\le r$, we may choose their lengths as
	\[
	t_1=r-\delta_1+1,
	\qquad
	t_2=\cdots=t_{\delta_1}=1,
	\]
	with only the first run when $\delta_1=1$.
	These lengths are positive and sum to $r$.
	
	\emph{Step 3: Constructing the exponent set.}
	Starting at position $0$, place these runs
	successively, inserting the prescribed internal
	gaps in nondecreasing order of length between
	successive runs. Their order does not affect
	the required gap-length counts.
	Let $T$ be the resulting set of occupied
	positions.
	
	By construction, $T$ contains $r$ elements,
	starts at $0$, and has $k-h$ missing positions
	between its endpoints. Hence its largest
	element is
	\[
	L=(r-1)+(k-h)=n-h-1.
	\]
	Thus $T$ has the required form and depends
	only on the prescribed loss sequence.
	The construction also covers $r=1$, in which
	case all losses equal $1$ and $T=\{0\}$.
	
	Now let $m>B_T$ and choose
	$\alpha\in\mathbb F_{q^m}$ of degree $m$.
	By Lemma~\ref{lem:fixed-pattern-mrd},
	the resulting code $C_{g_T}(n)$ is MRD,
	and $m\ge n$.
	
	Write
	\[
	D_j=\deg\operatorname{lcm}
	(g_T,\sigma(g_T),\ldots,\sigma^j(g_T)).
	\]
	The roots of $g_T$ lie in a single Frobenius
	orbit of length $m$, and $g_T$ has one
	multiplicity level. By
	Theorem~\ref{thm:orbit-gap-formula},
	$D_j-D_{j-1}$ equals the number of cyclic
	gaps whose lengths are at least $j$.
	
	Since $L<n\le m$, the exponents in $T$
	represent distinct positions in
	$\mathbb Z/m\mathbb Z$. Besides the prescribed
	internal gaps, there is exactly one external
	gap, running cyclically from $L$ back to $0$.
	Its length is
	\[
	m-L-1=m-n+h\ge h.
	\]
	Thus it contributes one to each increment
	$D_j-D_{j-1}$ for $1\le j\le h$.
	
	Combining the internal and external
	contributions, we obtain, for $1\le j\le h$,
	\[
	\begin{aligned}
		D_j-D_{j-1}
		&=1+\sum_{a=j}^{h-1}
		(\delta_a-\delta_{a+1})+(\delta_h-1)\\
		&=\delta_j.
	\end{aligned}
	\]
	Since $D_0=r$, summing these increments gives
	\[
	D_j
	=r+\sum_{a=1}^j\delta_a
	=r+k-d_j
	=n-d_j
	\qquad(0\le j\le h).
	\]
	
	Finally, formula~\eqref{eq:global-lcm}
	yields
	\[
	d_j(C_{g_T}(n),\sigma)
	=\max\{0,n-D_j\}
	=d_j
	\qquad(0\le j\le h).
	\]
	In particular, the intersection vanishes at
	$j=h$. Since the intersection spaces form
	a decreasing sequence, they remain zero for
	all $j>h$. This proves the prescribed profile.
\end{proof}

Since $e_i\le n-2$ and $k+i-1\le n-1$, the construction also
gives the preliminary uniform bound
\[
B_T\le(n-1)\sum_{e=0}^{n-2}q^e
=(n-1)\frac{q^{n-1}-1}{q-1}.
\]
Every extension degree larger than this bound realizes every
admissible profile. Section~\ref{sec:profile-field-bounds}
substantially improves it by refining the determinant argument.

\begin{example}
	\label{ex:profile-construction}
	Let $r=3$, $k=6$, and $n=9$. For losses $(3,2,1)$,
	Theorem~\ref{thm:universal-profile-realization} prescribes
	two internal gaps of lengths $1$ and $2$ between
	three singleton occupied runs. This gives the exponent
	set $T=\{0,2,5\}$ from Example~\ref{ex:gap-count}, with
	\[
	B_T=6+7q^2+8q^5.
	\]
	For every integer $m>B_T$ and every degree-$m$ parameter
	$\alpha$, the resulting code $C_{g_T}(9)$ is an $[9,6,4]$
	MRD code over $\mathbb F_{q^m}$, where the parameters
	denote length, dimension, and minimum rank distance.
	Its Frobenius intersection profile is $(6,3,1,0)$.
	Smaller extension bounds for this example are obtained
	in Section~\ref{sec:profile-field-bounds}.
\end{example}

\begin{remark}
	The necessary numerical conditions in
	Proposition~\ref{prop:necessary-profile} remain valid
	for any generator of
	$\operatorname{Gal}(\mathbb F_{q^m}/\mathbb F_q)$.
	In particular, $\sigma_s(x)=x^{q^s}$ is such a
	generator whenever $\gcd(s,m)=1$.
	
	The realization construction also extends to these
	automorphisms. Fix a positive integer $s$ and an
	admissible profile, and let
	$T=\{e_1<\cdots<e_r\}$ be the exponent set
	constructed in
	Theorem~\ref{thm:universal-profile-realization}.
	Define
	\[
	g_{T,s,\alpha}(X)
	=\prod_{i=1}^r(X-\alpha^{q^{s e_i}}),
	\qquad
	B_{T,s}=\sum_{i=1}^r(k+i-1)q^{s e_i}.
	\]
	For every $m>B_{T,s}$ with $\gcd(s,m)=1$,
	and every $\alpha\in\mathbb F_{q^m}$ of degree $m$,
	the argument of Lemma~\ref{lem:fixed-pattern-mrd}
	shows that $C_{g_{T,s,\alpha}}(n)$ is MRD.
	
	Moreover, $\sigma_s$ acts on the roots by advancing
	one position in the exponent pattern $T$.
	Thus the same gap-count argument gives the prescribed
	intersection profile with respect to $\sigma_s$.
	Consequently, every admissible profile is realized
	under $\sigma_s$ over all sufficiently large extensions
	whose degrees are coprime to $s$.
	
	The quantitative bounds in Section~\ref{sec:profile-field-bounds} and the
	every-extension threshold considered below are
	formulated for the ordinary $q$-Frobenius
	$\sigma=\sigma_1$.
\end{remark}

\section{Effective extension bounds and parameter density}
\label{sec:profile-field-bounds}

Fix a prime power $q$, integers $1\le r<n$, and an
exponent set
\[
T=\{0=e_1<\cdots<e_r=L\}\subseteq\mathbb Z_{\ge0},
\qquad k=n-r.
\]
We seek two types of bounds: an \emph{all-parameter bound},
ensuring that every degree-$m$ seed yields an MRD code,
and an \emph{existence bound}, ensuring that at least
one such seed works.

Following the subspace-indexed determinant approach
of \cite[Section~4, Proposition~4.2]{NeriEtAl2018},
we remove the ordinary Moore factor, which is nonzero
on the test subspaces when $m\ge n$. This improves
the direct degree bound $B_T$ to $\Delta_T$ and leads
to the counting budgets $S_T$ and $S_T(m)$.
Under $m\ge n$ and $m>L$, the sufficient conditions
are $m>\Delta_T$ for every degree-$m$ seed and
$D_q(m)>S_T(m)$ for at least one such seed.
The latter yields the uniform extension bounds below.

\subsection{Moore quotients and exact degrees}

To bound these degrees, define
\begin{equation}
	w_i=q^{e_i}-q^{i-1},
	\qquad
	\Delta_T=\sum_{i=1}^r(k+i-1)w_i.
	\label{eq:profile-quotient-degree}
\end{equation}
Since $e_i\ge i-1$, we have $w_i\ge0$.
The lemma below shows that $\Delta_T$ bounds the
degrees of the specialized Moore quotients.

The underlying Moore divisibility is classical
\cite{BartoliZhou2020}. Here we obtain an exact
degree formula in terms of the leading degrees
of the polynomial subspaces. We use reduced
echelon bases $P_1,\ldots,P_r$ with
$d_1=\deg P_1<\cdots<d_r=\deg P_r$,
where each $P_i$ is monic and has zero
coefficient at $Z^{d_j}$ whenever $j\ne i$.

\begin{lemma}
\label{lem:remove-moore-factor}
Let $U$ be an $r$-dimensional $\mathbb F_q$-subspace of
$\mathbb F_q[Z]_{<n}$. For any ordered basis $P_1,\ldots,P_r$ of $U$,
the quotient
\begin{equation}
 Q_{T,U}(Z)=
 \frac{\det(P_j(Z)^{q^{e_i}})_{i,j=1}^r}
      {\det(P_j(Z)^{q^{i-1}})_{i,j=1}^r}
 \label{eq:moore-specialized-quotient}
\end{equation}
is a nonzero polynomial in $\mathbb F_q[Z]$, independent of the
choice of basis. If $d_1<\cdots<d_r$ are the successive leading
degrees in an echelon basis of $U$, then
\begin{equation}
 \deg Q_{T,U}=\sum_{i=1}^r d_iw_i\leq\Delta_T.
 \label{eq:exact-quotient-degree}
\end{equation}
If $\alpha$ has degree at least $n$ over $\mathbb F_q$, the
denominator in \eqref{eq:moore-specialized-quotient} is nonzero at
$\alpha$.
\end{lemma}

\begin{proof}
	We first establish the polynomial divisibility.
	Introduce independent variables $X_1,\ldots,X_r$ and write
	\[
	G_T(\mathbf X)
	=\det(X_j^{q^{e_i}})_{i,j=1}^r,
	\qquad
	M(\mathbf X)
	=\det(X_j^{q^{i-1}})_{i,j=1}^r.
	\]
	Let $\ell=\sum_{j=1}^r c_jX_j$ be a nonzero
	$\F_q$-linear form. Modulo $\ell$, the variables
	$X_1,\ldots,X_r$ satisfy a nontrivial
	$\F_q$-linear relation. Since Frobenius powers
	preserve such relations, both determinants
	vanish modulo $\ell$. Thus $\ell$ divides
	both $G_T$ and $M$.
	
	Up to nonzero scalar multiples, there are
	\[
	\frac{q^r-1}{q-1}
	=1+q+\cdots+q^{r-1}
	\]
	distinct nonzero $\F_q$-linear forms.
	They are pairwise relatively prime, and their
	number equals the total degree of $M$.
	Moreover, $M$ is nonzero, since its determinant
	expansion has a monomial arising from a unique
	permutation. Hence their product is a nonzero
	scalar multiple of $M$. As each of these
	forms also divides $G_T$, we conclude that
	$M$ divides $G_T$ in
	$\F_q[X_1,\ldots,X_r]$.
	
	Now choose an echelon basis
	$P_1,\ldots,P_r$ of $U$ with leading degrees
	$0\le d_1<\cdots<d_r<n$.
	By the strict leading-degree argument of
	Lemma~\ref{lem:fixed-pattern-mrd},
	both specialized determinants are nonzero,
	and
	\[
	\begin{aligned}
		\deg G_T(P_1(Z),\ldots,P_r(Z))
		&=\sum_{i=1}^r d_iq^{e_i},\\
		\deg M(P_1(Z),\ldots,P_r(Z))
		&=\sum_{i=1}^r d_iq^{i-1}.
	\end{aligned}
	\]
	In particular, the denominator in
	\eqref{eq:moore-specialized-quotient} is not identically
	zero. Specializing the polynomial identity
	$G_T=MR_T$, where
	$R_T\in\F_q[X_1,\ldots,X_r]$, therefore gives
	\[
	Q_{T,U}(Z)
	=R_T(P_1(Z),\ldots,P_r(Z))
	\in\F_q[Z].
	\]
	Subtracting the two determinant degrees yields
	\[
	\deg Q_{T,U}
	=\sum_{i=1}^r d_i
	\bigl(q^{e_i}-q^{i-1}\bigr)
	=\sum_{i=1}^r d_iw_i
	\le\Delta_T,
	\]
	where the inequality follows from
	$d_i\le k+i-1$ and $w_i\ge0$.
	
	The quotient is independent of the chosen
	basis of $U$. Indeed, a change of basis
	by a matrix in $\GL_r(q)$ multiplies
	both determinants by the same nonzero
	determinant, leaving their quotient unchanged.
	
	Finally, if $\alpha$ has degree at least
	$n$ over $\F_q$, evaluation at $\alpha$
	is injective on $\F_q[Z]_{<n}$.
	Thus $P_1(\alpha),\ldots,P_r(\alpha)$
	are $\F_q$-linearly independent, and their
	ordinary Moore determinant is nonzero.
	This proves the final assertion.
\end{proof}

\subsection{Counting successful parameters}

Grouping determinant degrees by subspace data already appears in
\cite[Lemma~4.5 and Theorem~4.6]{NeriEtAl2018}.
We group the specialized Moore quotients by the leading degrees
of their polynomial bases, using the exact degrees in
Lemma~\ref{lem:remove-moore-factor}.
For each sequence
\[
\boldsymbol d=(d_1,\ldots,d_r),
\qquad 0\le d_1<\cdots<d_r<n,
\]
define
\[
c(\boldsymbol d)
=q^{\sum_{i=1}^r(d_i-i+1)},
\qquad
b_T(\boldsymbol d)=\sum_{i=1}^r d_iw_i.
\]
The quantity $c(\boldsymbol d)$ counts the
$r$-dimensional subspaces whose reduced echelon
basis has leading degrees $d_1,\ldots,d_r$.
Indeed, below its leading term the $i$-th basis polynomial
has $d_i$ coefficient positions, of which the $i-1$ earlier
pivot positions are fixed at zero. Its $d_i-i+1$ remaining
coefficients are free.
By Lemma~\ref{lem:remove-moore-factor},
$b_T(\boldsymbol d)$ is the exact degree of
the determinant quotient $Q_{T,U}$ for every
subspace with this leading-degree sequence.

We introduce two degree budgets:
\begin{align}
	S_T
	&=\sum_{0\le d_1<\cdots<d_r<n}
	c(\boldsymbol d)b_T(\boldsymbol d),
	\label{eq:schubert-budget}\\
	S_T(m)
	&=m\sum_{0\le d_1<\cdots<d_r<n}
	c(\boldsymbol d)
	\left\lfloor\frac{b_T(\boldsymbol d)}m\right\rfloor.
	\label{eq:degree-m-budget}
\end{align}
The first budget sums the degrees of all
determinant quotients, one for each
$r$-dimensional subspace of $\F_q^n$.

The second budget accounts only for roots of degree exactly
$m$ over $\F_q$: a nonzero polynomial of degree $b$ over $\F_q$
has at most $m\lfloor b/m\rfloor$ such roots. Thus it gives a potentially
smaller bound on the parameters excluded by the determinant tests.

By construction,
\begin{equation}
	S_T(m)\le S_T
	\le\qbinom{n}{r}\Delta_T,
	\label{eq:budget-upper-bounds}
\end{equation}
where $\qbinom{n}{r}$ denotes the number of
$r$-dimensional subspaces of $\F_q^n$.

Finally, the number of elements of degree
exactly $m$ over $\F_q$ in $\F_{q^m}$ is
\[
D_q(m)=\sum_{d\mid m}\mu(d)q^{m/d},
\]
where $\mu$ is the number-theoretic
M\"obius function.
These are the candidate parameters for
the construction. The next theorem compares
their number with the degree budgets to
guarantee the existence of MRD realizations.

\begin{theorem}
\label{thm:pattern-proportion}
Assume $m\ge n$ and $m>L$. Of the $D_q(m)$ elements
$\alpha\in\mathbb F_{q^m}$ having degree $m$ over $\mathbb F_q$,
at least
\begin{equation}
 \max\{0,D_q(m)-S_T(m)\}
 \label{eq:good-alpha-count}
\end{equation}
generate an MRD code through
\[
 g_{T,\alpha}(X)=\prod_{i=1}^r(X-\alpha^{q^{e_i}}).
\]
Consequently, their proportion is at least
\begin{equation}
 \max\left\{0,1-\frac{S_T(m)}{D_q(m)}\right\}
 \geq
 \max\left\{0,1-\frac{S_T}{D_q(m)}\right\}.
 \label{eq:good-alpha-proportion}
\end{equation}
In particular, \emph{at least one} degree-$m$ parameter works if
\begin{equation}
 D_q(m)>S_T(m).
 \label{eq:exact-degree-existence}
\end{equation}
The simpler conditions $D_q(m)>S_T$ or
$D_q(m)>\qbinom{n}{r}\Delta_T$ are also sufficient.
In contrast, \emph{every} degree-$m$ parameter works if
\begin{equation}
 m\ge n,\qquad m>L,\qquad m>\Delta_T.
 \label{eq:all-alpha-quotient-bound}
\end{equation}
Here $\Delta_T=B_T-\sum_i(k+i-1)q^{i-1}$ improves the
all-parameter bound of Lemma~\ref{lem:fixed-pattern-mrd}, while
the counting criterion can require a much smaller extension degree.
\end{theorem}

\begin{proof}
	Let $\alpha\in\F_{q^m}$ have degree $m$ over
	$\F_q$. Since $m>L$, the roots of $g_{T,\alpha}$
	are distinct, and evaluation at these roots
	gives a rank-$r$ parity-check matrix $H$, as in
	Lemma~\ref{lem:fixed-pattern-mrd}.
	
	For any rank-$r$ matrix $A\in\F_q^{n\times r}$,
	let $P_1,\ldots,P_r$ be the independent
	polynomials corresponding to its columns, and put
	$U=\langle P_1,\ldots,P_r\rangle_{\F_q}$.
	Then
	\[
	\det(HA)
	=\det(P_j(\alpha)^{q^{e_i}})_{i,j=1}^r.
	\]
	Since $m\ge n$, Lemma~\ref{lem:remove-moore-factor}
	ensures that the ordinary Moore determinant
	$\det(P_j(\alpha)^{q^{i-1}})_{i,j=1}^r$
	is nonzero. Consequently,
	\[
	\det(HA)=0
	\quad\Longleftrightarrow\quad
	Q_{T,U}(\alpha)=0.
	\]
	As $Q_{T,U}$ is independent of the chosen basis
	of $U$, the MRD criterion reduces to testing
	one quotient for each $r$-dimensional subspace
	$U\le\F_q[Z]_{<n}$.
	
	We now count the parameters excluded by these
	tests. For each leading-degree sequence
	$\boldsymbol d$, there are $c(\boldsymbol d)$
	subspaces with that sequence, and the
	corresponding quotients have degree
	$b_T(\boldsymbol d)$. Hence
	\[
	S_T
	=\sum_{U\in\Gr(r,\F_q^n)}\deg Q_{T,U}
	\le\qbinom{n}{r}\Delta_T.
	\]
	
	More precisely, a nonzero polynomial
	$Q\in\F_q[Z]$ of degree $b$ has at most
	$\lfloor b/m\rfloor$ distinct irreducible
	factors of degree $m$. Each such factor has
	exactly $m$ roots in $\F_{q^m}$, so $Q$
	vanishes at no more than
	$m\lfloor b/m\rfloor$ elements of degree $m$.
	Applying this estimate to all the quotients
	and taking their union, we find that at most
	$S_T(m)$ degree-$m$ parameters fail the MRD
	test. Possible overlaps among their root sets
	only improve this bound.
	
	Since there are $D_q(m)$ elements of degree
	$m$ over $\F_q$, at least
	\[
	\max\{0,D_q(m)-S_T(m)\}
	\]
	of them yield MRD codes. This proves
	\eqref{eq:good-alpha-count}--\eqref{eq:exact-degree-existence},
	and \eqref{eq:budget-upper-bounds}
	gives the simpler sufficient conditions.
	
	Finally, if $m>\Delta_T$, every quotient
	$Q_{T,U}$ is a nonzero polynomial of degree
	less than $m$. It therefore cannot vanish at
	any element of degree $m$. Thus every
	degree-$m$ parameter yields an MRD code,
	proving \eqref{eq:all-alpha-quotient-bound}.
\end{proof}

\begin{remark}
	Theorem~\ref{thm:pattern-proportion} gives lower bounds
	on the number and proportion of successful degree-$m$
	parameters, rather than counting distinct codes or
	rank-equivalence classes. Its tests concern the chosen short power-basis
	space; no Moore property on arbitrary independent tuples
	in the extension field is assumed.
\end{remark}

\subsection{Uniform extension bounds and density}

\begin{corollary}
\label{cor:polynomial-profile-bound}
Let $r,k\ge2$, $n=r+k$, and prescribe any admissible loss sequence
\[
 r\ge\delta_1\ge\cdots\ge\delta_h\ge1,
 \qquad \sum_{j=1}^h\delta_j=k.
\]
Choose an exponent set $T$ by the root-run construction of
Theorem~\ref{thm:universal-profile-realization}, so that $e_1=0$
and its largest exponent is $L=n-h-1$.
For this \emph{particular} profile and exponent set, there is an MRD
realization over every $\mathbb F_{q^m}$ satisfying
\begin{equation}
 m\ge\max\left\{n,
       \left\lceil\log_q(2S_T)\right\rceil\right\}.
 \label{eq:logarithmic-profile-bound}
\end{equation}
When $S_T=0$, the condition is understood simply as $m\ge n$.

We also give bounds that do not depend on the profile or the
choice of $T$. Set
\[
 h_0=\left\lceil\frac{k}{r}\right\rceil,\qquad
 H_{q,n,r}=\sum_{i=2}^r(k+i-1)q^{i-1},\qquad
 A_{q,n,r}=(q^{k-h_0}-1)H_{q,n,r}.
\]
Here $h_0$ is a lower bound for the number $h$ of positive losses.
Every admissible profile is realized over \emph{each} extension with
\begin{equation}
 m\ge\max\left\{n,
 \left\lceil\log_q\left(2\qbinom{n}{r} A_{q,n,r}\right)\right\rceil
 \right\}.
 \label{eq:uniform-logarithmic-bound}
\end{equation}
In particular, the simpler condition
\begin{equation}
 m\ge\max\left\{n,
 rk+n-h_0+
 \left\lceil\log_q\left(\frac{8(n-1)}{q-1}\right)\right\rceil
 \right\}
 \label{eq:polynomial-extension-bound}
\end{equation}
is sufficient. For fixed $q$, this is an extension degree of
order $O(rk+n+\log n)$, in contrast to the exponential
all-parameter bound of Section~\ref{sec:profile-realization}.
\end{corollary}

\begin{proof}
	\emph{The bound for a fixed profile.}
	For the exponent set $T$ constructed in
	Theorem~\ref{thm:universal-profile-realization},
	the gap calculation gives the prescribed
	intersection profile for every $m\ge n$:
	the external gap has length $m-n+h\ge h$.
	It remains to ensure the MRD property.
	
	Since $r,k\ge2$, we have $m\ge n\ge4$.
	Every element of $\F_{q^m}$ of degree less
	than $m$ lies in a proper subfield, whose
	degree over $\F_q$ is at most
	$\lfloor m/2\rfloor$. Consequently,
	\[
	\begin{aligned}
		q^m-D_q(m)
		&\le \sum_{d=1}^{\lfloor m/2\rfloor}q^d\\
		&<\frac{q^{\lfloor m/2\rfloor+1}}{q-1}\\
		&\le\frac{q^{m-1}}{q-1}
		\le\frac12q^m.
	\end{aligned}
	\]
	Hence
	\begin{equation}
		D_q(m)>\frac12q^m.
		\label{eq:full-degree-half}
	\end{equation}
	
	If $S_T>0$ and
	$m\ge\lceil\log_q(2S_T)\rceil$, then
	\[
	D_q(m)>\frac12q^m\ge S_T\ge S_T(m).
	\]
	Theorem~\ref{thm:pattern-proportion}
	therefore guarantees a degree-$m$ parameter
	yielding an MRD code. If $S_T=0$, all
	determinant quotients are nonzero constants,
	so every degree-$m$ parameter works.
	This proves
	\eqref{eq:logarithmic-profile-bound}.
	
	\emph{A bound uniform in the profile.}
	Since the positive losses are at most $r$
	and sum to $k$,
	\[
	k=\sum_{j=1}^h\delta_j\le rh,
	\qquad
	h\ge h_0=\left\lceil\frac{k}{r}\right\rceil.
	\]
	The exponent set
	$T=\{e_1<\cdots<e_r\}$ ends at
	$L=n-h-1$. Since there are $r-i$
	distinct exponents greater than $e_i$,
	\[
	e_i\le L-(r-i)=k-h+i-1.
	\]
	Thus $w_1=0$, while for $i\ge2$,
	\[
	w_i=q^{e_i}-q^{i-1}
	\le(q^{k-h}-1)q^{i-1}.
	\]
	It follows that
	\[
	\begin{aligned}
		\Delta_T
		&=\sum_{i=1}^r(k+i-1)w_i\\
		&\le(q^{k-h}-1)H_{q,n,r}\\
		&\le(q^{k-h_0}-1)H_{q,n,r}
		=A_{q,n,r}.
	\end{aligned}
	\]
	Consequently,
	\[
	S_T
	\le\qbinom{n}{r}\Delta_T
	\le\qbinom{n}{r}A_{q,n,r}.
	\]
	Substituting this profile-independent
	estimate into the fixed-profile bound proves
	\eqref{eq:uniform-logarithmic-bound}.
	
	\emph{The closed-form estimate.}
	Since $w_i\le q^{e_i}$,
	$k+i-1\le n-1$, and the exponents are
	distinct elements of $\{0,\ldots,L\}$,
	\[
	\begin{aligned}
		\Delta_T
		&=\sum_{i=1}^r(k+i-1)w_i\\
		&<(n-1)\sum_{j=0}^{L}q^j\\
		&<(n-1)\frac{q^{L+1}}{q-1}.
	\end{aligned}
	\]
	Using the Gaussian binomial estimate
	$\qbinom{n}{r}<4q^{rk}$, justified below,
	and $L+1=n-h$, we obtain
	\[
	\begin{aligned}
		2S_T
		&\le2\qbinom{n}{r}\Delta_T\\
		&<\frac{8(n-1)}{q-1}q^{rk+n-h}\\
		&\le\frac{8(n-1)}{q-1}q^{rk+n-h_0}.
	\end{aligned}
	\]
	Condition~\eqref{eq:polynomial-extension-bound}
	ensures that $q^m$ is at least the final
	upper bound. Since the preceding inequality
	is strict, we have $2S_T<q^m$, and hence
	\[
	D_q(m)>\frac12q^m>S_T\ge S_T(m).
	\]
	Theorem~\ref{thm:pattern-proportion}
	therefore gives an MRD realization.
	
	Finally, the Gaussian estimate follows
	from its product formula. For $q\ge2$,
	\[
	\begin{aligned}
		\prod_{j\ge1}(1-q^{-j})
		&\ge\prod_{j\ge1}(1-2^{-j})\\
		&\ge\frac12
		\left(1-\sum_{j\ge2}2^{-j}\right)
		=\frac14.
	\end{aligned}
	\]
	Since the numerator in the Gaussian
	binomial product formula is strictly
	less than $1$, this gives
	$\qbinom{n}{r}<4q^{rk}$, completing
	the proof.
\end{proof}

\begin{corollary}
	\label{cor:profile-parameter-density}
	Under the hypotheses of
	Corollary~\ref{cor:polynomial-profile-bound},
	let $\rho_{T,m}$ denote the proportion of elements
	$\alpha\in\mathbb F_{q^m}$ of degree $m$ over
	$\mathbb F_q$ for which $C_{g_{T,\alpha}}(n)$
	is MRD and has the prescribed profile.
	Then, for every $m\ge n$,
	\begin{equation}
		\begin{aligned}
			\rho_{T,m}
			&\ge
			\max\left\{0,1-\frac{S_T(m)}{D_q(m)}\right\}\\
			&\ge
			\max\left\{0,1-
			\frac{8(n-1)}{q-1}q^{rk+n-h-m}\right\}.
		\end{aligned}
		\label{eq:profile-density-explicit}
	\end{equation}
	
	In particular, for any $0<\varepsilon<1$,
	the condition
	\[
	m\ge
	\max\left\{
	n,\,
	rk+n-h+
	\left\lceil
	\log_q\left(
	\frac{8(n-1)}{(q-1)\varepsilon}
	\right)
	\right\rceil
	\right\}
	\]
	ensures that $\rho_{T,m}\ge1-\varepsilon$.
	Replacing $h$ by $h_0=\lceil k/r\rceil$
	gives a sufficient bound uniform over all
	admissible profiles. The parameter $\alpha$
	may be chosen separately for each profile.
\end{corollary}

\begin{proof}
	Theorem~\ref{thm:pattern-proportion} gives the
	first inequality in \eqref{eq:profile-density-explicit}.
	For $m\ge n\ge4$, the estimate
	$D_q(m)>\frac12q^m$ from
	\eqref{eq:full-degree-half}, together with
	$S_T(m)\le S_T$ and Corollary~
	\ref{cor:polynomial-profile-bound}, yields
	\[
	\frac{S_T(m)}{D_q(m)}
	\le\frac{S_T}{D_q(m)}
	<
	\frac{8(n-1)}{q-1}q^{rk+n-h-m}.
	\]
	This proves the second inequality in
	\eqref{eq:profile-density-explicit}.
	
	Under the stated condition on $m$, the
	last expression is at most $\varepsilon$.
	Consequently,
	\[
	\rho_{T,m}
	\ge 1-\frac{S_T(m)}{D_q(m)}
	\ge 1-\varepsilon.
	\]
	Finally, since $h\ge h_0$, replacing $h$
	by $h_0$ gives a sufficient condition
	uniform over all admissible profiles.
\end{proof}

\begin{example}
	\label{ex:binary-profile-budget}
	For $T=\{0,2,5\}$, $n=9$, $r=3$, $k=6$, and $q=2$
	from Example~\ref{ex:profile-construction}, the degree
	bounds and budgets are
	\begin{center}
	\begin{tabular}{@{}lr@{}}
	\hline
	Quantity & Value\\
	\hline
	$B_T=6+7\cdot4+8\cdot32$ & $290$\\
	$\Delta_T$ & $238$\\
	$\genfrac{[}{]}{0pt}{}{9}{3}_{2}$ & $788035$\\
	$S_T$ & $183724894$\\
	$D_2(28)$ & $268419060$\\
	$S_T(28)$ & $173424944$\\
	\hline
	\end{tabular}
	\end{center}
	The budgets $S_T$ and $S_T(28)$ follow by summing
	\eqref{eq:schubert-budget}--\eqref{eq:degree-m-budget}
	over leading-degree triples. Lemma~\ref{lem:fixed-pattern-mrd}
	guarantees every degree-$m$ parameter for $m\ge291$;
	\eqref{eq:all-alpha-quotient-bound} improves this to $m\ge239$.
	At $m=28$, Theorem~\ref{thm:pattern-proportion} guarantees
	$94994116$ successful parameters, giving
	\[
	\rho_{T,28}\ge\frac{94994116}{268419060}>0.35.
	\]
	This is an existence and positive-density bound over
	$\mathbb F_{2^{28}}$, not a claim that every degree-$28$
	parameter works.
\end{example}

\begin{corollary}
	\label{cor:boundary-threshold}
	If $\min\{r,k\}=1$, then
	\[
	M_{\mathrm{univ}}(q,n,r)=n.
	\]
\end{corollary}

\begin{proof}
	In either boundary case, the unique admissible loss sequence
	consists of $k$ ones, giving $T=\{0,1,\ldots,r-1\}$.
	Thus $e_i=i-1$ and $w_i=0$ for every $i$: the generalized
	and ordinary Moore determinants coincide, so $Q_{T,U}=1$.
	Theorem~\ref{thm:pattern-proportion} gives an MRD realization
	with the prescribed profile for every degree-$m$ parameter
	when $m\ge n$. Since $M_{\mathrm{univ}}$ is defined to be
	at least $n$, equality follows.
\end{proof}

The next section establishes lower bounds
for the every-extension threshold.
In particular,
Theorem~\ref{thm:profile-uniform-bound-near-sharp}
shows that, along $r=s$ and $n=s^2+2$,
the upper and lower bounds differ by only
$O_q(\log n)$.

\section{A near-sharp obstruction for every-extension bounds}
\label{sec:profile-lower-bounds}

To bound the every-extension threshold from below, we
must find an extension degree at which some admissible
profile has no polynomial-multiple MRD realization.
Excluding only a particular exponent set or parameter
would not suffice. The following lemma determines all
MRD generators with a specially chosen profile, without
assuming that they are squarefree or split over $E$.

\begin{lemma}
	\label{lem:ap-profile-all-points}
	Let $r,k\ge2$, put $n=r+k$, and choose an integer
	$e\ge2$ such that $n\ge re+2$. Define
	\[
	a=k-r(e-1)\ge2,
	\qquad
	h=e-1+a=k-(r-1)(e-1).
	\]
	Let $E=\mathbb F_{q^m}$ with $m\ge n$, and suppose
	that a monic polynomial $g\in E[X]$ of degree $r$
	generates an MRD code $C_g(n)$ whose positive
	$q$-Frobenius losses are
	\begin{equation}
		(\delta_1,\ldots,\delta_h)
		=
		\bigl(
		\underbrace{r,\ldots,r}_{e-1\text{ entries}},
		\underbrace{1,\ldots,1}_{a\text{ entries}}
		\bigr).
		\label{eq:ap-profile-obstructed-losses}
	\end{equation}
	Then there exists a unique $\beta\in E$ such that
	\begin{equation}
		g(X)=\prod_{i=0}^{r-1}
		(X-\beta^{q^{ie}}).
		\label{eq:ap-profile-seed-representation}
	\end{equation}
	Moreover, its degree
	$t=[\mathbb F_q(\beta):\mathbb F_q]$
	satisfies $t\mid m$ and $t\ge n$.
	In particular, every MRD generator with this
	profile splits over $E$ and is squarefree.
	
	Conversely, if $\beta\in E$ has degree $t\ge n$,
	then the polynomial in
	\eqref{eq:ap-profile-seed-representation}
	has the prescribed intersection profile.
	This converse concerns the profile only;
	it does not assert the MRD property.
\end{lemma}

\begin{proof}
	Write $\sigma(x)=x^q$, acting coefficientwise on
	polynomials, and put
	\[
	L_j=\operatorname{lcm}
	(g,\sigma(g),\ldots,\sigma^j(g)),
	\qquad D_j=\deg L_j.
	\]
	For $0\le j<h$, the prescribed intersection
	dimensions are positive. Hence
	formula~\eqref{eq:global-lcm} gives
	$D_j=n-d_j$, and therefore
	\[
	D_j-D_{j-1}=\delta_j
	\qquad(1\le j<h).
	\]
	At $j=h$, the vanishing $d_h=0$ gives only
	$D_h\ge n$; the identities above deliberately
	stop before this truncation point.
	Since $a\ge2$, the last two positive losses
	are both $1$. In particular,
	\begin{equation}
		D_{h-1}=n-1,
		\qquad
		D_{h-1}-D_{h-2}=1.
		\label{eq:ap-profile-last-preterminal-degrees}
	\end{equation}
	We use these identities to determine the
	factorization of $g$.
	
	\emph{Step 1: Excluding complete factor orbits.}
	Since $D_{h-1}=n-1<n$,
	Corollary~\ref{cor:no-base-factor-before-vanishing}
	shows that $L_{h-1}$ has no nonconstant factor
	in $\mathbb F_q[X]$.
	
	Consider an occupied $\sigma$-orbit of monic
	irreducible factors over $E$,
	\[
	\mathcal O=(P_0,\ldots,P_{t-1}),
	\]
	and write
	$e_{\mathcal O}=\deg P_i$ and
	$b_i=v_{P_i}(g)$.
	At each nonempty multiplicity level $u$,
	the corresponding superlevel set
	\[
	S_{\mathcal O,u}
	=\{i\in\mathbb Z/t\mathbb Z:b_i\ge u\}
	\]
	is enlarged in $L_j$ to
	$S_{\mathcal O,u}+\{0,\ldots,j\}$.
	
	Suppose that this enlargement covered the
	entire orbit at $j=h-1$.
	Since
	$S_{\mathcal O,u}\subseteq S_{\mathcal O,1}$,
	the first-level enlargement would also
	cover the entire orbit. Consequently,
	\[
	\prod_{i=0}^{t-1}P_i
	\]
	would divide $L_{h-1}$.
	But this product is a nonconstant polynomial
	fixed by $\sigma$, hence belongs to
	$\mathbb F_q[X]$, a contradiction.
	
	Therefore no nonempty superlevel set can
	cover the entire orbit after the first
	$h-1$ shifts. By the cyclic gap formula,
	each such set must have at least one
	complementary gap of length at least $h$.
	
	\emph{Step 2: Forcing a single linear-factor orbit.}
	By Theorem~\ref{thm:orbit-gap-formula}
	and \eqref{eq:ap-profile-last-preterminal-degrees},
	the weighted gap counts satisfy
	\[
	\sum_{\ell\ge h-1}\nu_\ell
	=D_{h-1}-D_{h-2}=1.
	\]
	On the other hand, Step 1 shows that every
	nonempty superlevel set has at least one
	gap of length at least $h$, contributing
	the positive integer weight $e_{\mathcal O}$
	to this sum.
	
	Since the total weighted contribution is $1$,
	there can be exactly one nonempty superlevel set,
	and its factor degree must be $e_{\mathcal O}=1$.
	This set has exactly one gap of length at least
	$h-1$, which in fact has length at least $h$
	by Step 1.
	
	Consequently, $g$ has exactly one occupied
	factor orbit and one nonempty multiplicity
	level. Its irreducible factors are therefore
	distinct and linear, so $g$ is squarefree
	and splits over $E$, with all $r$ roots
	belonging to a single $q$-Frobenius orbit.
	
	\emph{Step 3: Recovering the arithmetic progression.}
	The unique nonempty superlevel set now
	consists of $r$ occupied positions in
	one cyclic orbit, with a unique gap of
	length at least $h-1$.
	
	For $1\le\ell\le h-2$, the gap inversion
	formula gives
	\[
	\nu_\ell=\delta_\ell-\delta_{\ell+1}.
	\]
	Since $a\ge2$, we have $e-1\le h-2$.
	In particular, $\delta_{e-1}=r$ and
	$\delta_e=1$. The prescribed loss sequence
	therefore yields
	\[
	\nu_{e-1}=r-1,
	\qquad
	\nu_\ell=0
	\quad
	(1\le\ell\le h-2,\ \ell\ne e-1).
	\]
	
	All gaps other than the unique long gap
	have length at most $h-2$.
	Hence there are exactly $r-1$ such gaps,
	each of length $e-1$.
	
	There are exactly $r$ inter-occupied spacings
	around the orbit, allowing zero-length spacings
	between adjacent occupied positions.
	The unique long gap and the $r-1$ gaps
	of length $e-1>0$ already account for all
	these spacings. Hence no zero-length spacing
	remains.
	
	Cutting the orbit at the long gap, we find
	that every two successive occupied positions
	are separated by exactly $e-1$ missing
	positions.
	
	Let $\beta$ be the root at the first occupied
	position following the long gap.
	The remaining roots are obtained by advancing
	successively by $e$ Frobenius positions.
	Their exponent positions are therefore
	\[
	T=\{0,e,2e,\ldots,(r-1)e\},
	\]
	which proves
	\eqref{eq:ap-profile-seed-representation}.
	
	Let $t$ denote the length of this Frobenius
	orbit, and let $\ell_{\mathrm{ext}}$ be the
	length of its unique long gap.
	Since $\ell_{\mathrm{ext}}\ge h$,
	\[
	\begin{aligned}
		t
		&=(r-1)e+1+\ell_{\mathrm{ext}}\\
		&\ge(r-1)e+1+h=n.
	\end{aligned}
	\]
	Because the orbit consists of the
	$q$-Frobenius conjugates of $\beta$, we have
	\[
	t=[\mathbb F_q(\beta):\mathbb F_q].
	\]
	As $\beta\in E=\mathbb F_{q^m}$,
	it follows that $t\mid m$.
	
	\emph{Uniqueness of the seed.}
	The long gap is uniquely distinguished,
	since its length is at least $h$, whereas
	all other gaps have length $e-1<h$.
	If another seed $\beta'$ gave the same
	polynomial, then $\beta'=\beta^{q^c}$
	for some $c$, and translation by $c$
	would preserve $T$ modulo $t$.
	Such a translation must preserve the
	unique long gap and hence fix its first
	succeeding occupied position.
	Therefore $c\equiv0\pmod t$ and
	$\beta'=\beta$, proving uniqueness.
	
	\emph{Conversely,}
	let $\beta\in E$ have degree $t\ge n$
	and define $g$ by
	\eqref{eq:ap-profile-seed-representation}.
	All exponents in $T+\{0,\ldots,h\}$
	lie between $0$ and
	\[
	(r-1)e+h=n-1<t.
	\]
	Thus distinct exponents correspond to
	distinct Frobenius conjugates.
	
	For $0\le j\le h$, direct counting gives
	\[
	|T+\{0,\ldots,j\}|
	=
	\begin{cases}
		r(j+1),
		&0\le j\le e-1,\\
		(r-1)e+j+1,
		&e\le j\le h.
	\end{cases}
	\]
	These cardinalities equal $D_j(g)$.
	Consequently,
	\[
	D_j-D_{j-1}
	=
	\begin{cases}
		r,&1\le j\le e-1,\\
		1,&e\le j\le h.
	\end{cases}
	\]
	Moreover, $D_h=n$.
	Formula~\eqref{eq:global-lcm} therefore
	gives precisely the loss sequence
	\eqref{eq:ap-profile-obstructed-losses},
	and the intersections remain zero for
	$j>h$.
	This completes the proof.
\end{proof}

Crucially, Lemma~\ref{lem:ap-profile-all-points} makes no
assumption about the factorization of $g$. The obstruction
below therefore excludes every polynomial-multiple MRD
realization of the chosen profile, not just the root-pattern
construction of Section~\ref{sec:profile-realization}.

\begin{theorem}
	\label{thm:profile-uniform-bound-near-sharp}
	Let $r,k\ge2$, put $n=r+k$, and let $e$ be an
	integer satisfying
	\[
	2\le e\le r,
	\qquad n\ge re+2.
	\]
	Then
	\begin{equation}
		M_{\mathrm{univ}}(q,n,r)\ge e(n-1)+1.
		\label{eq:profile-uniform-lower-bound}
	\end{equation}
	In particular, if $r\ge2$ and $k\ge r+2$,
	then
	\[
	M_{\mathrm{univ}}(q,n,r)\ge2n-1.
	\]
	
	More precisely, for every integer $s\ge2$, take
	\[
	r=s,\qquad n=s^2+2,\qquad k=s^2-s+2.
	\]
	Then
	\begin{align}
		s^3+s+1
		&\le M_{\mathrm{univ}}(q,s^2+2,s)
		\le s^3+s+2+
		\left\lceil
		\log_q\!\left(\frac{8(s^2+1)}{q-1}\right)
		\right\rceil.
		\label{eq:profile-uniform-near-sharp-sequence}
	\end{align}
	For fixed $q$, the upper and lower bounds
	differ by only $O_q(\log n)$.
	In particular, no uniform bound of order
	$O_q(n)$ can guarantee realization of every
	admissible profile over every larger extension
	within the polynomial-multiple family.
\end{theorem}

\begin{proof}
	We use the loss sequence
	\eqref{eq:ap-profile-obstructed-losses}
	from Lemma~\ref{lem:ap-profile-all-points}.
	It consists of $e-1$ entries equal to $r$,
	followed by
	\[
	a=k-r(e-1)\ge2
	\]
	entries equal to $1$.
	The losses are positive and nonincreasing,
	their largest value is $r$, and their sum is
	$r(e-1)+a=k$. Thus the sequence is admissible.
	
	Set
	\[
	m_0=e(n-1).
	\]
	We show that this profile cannot be realized
	by any polynomial-multiple MRD code over
	$E=\mathbb F_{q^{m_0}}$.
	
	Suppose, to the contrary, that such a code
	$C_g(n)$ exists. By
	Lemma~\ref{lem:ap-profile-all-points},
	its generator must have the form
	\[
	g(X)=\prod_{i=0}^{r-1}
	(X-\beta^{q^{ie}})
	\]
	for some $\beta\in E$ whose degree
	\[
	t=[\mathbb F_q(\beta):\mathbb F_q]
	\]
	satisfies $t\mid m_0$ and $t\ge n$.
	
	Put
	\[
	d=\gcd(e,t),
	\qquad
	K=\mathbb F_{q^d},
	\qquad
	\ell=\frac{t}{d}.
	\]
	Since $t\mid e(n-1)$, we have
	\[
	\frac{t}{d}\mid
	\frac{e}{d}(n-1).
	\]
	But
	\[
	\gcd\left(\frac{t}{d},
	\frac{e}{d}\right)=1,
	\]
	so
	\begin{equation}
		\ell\mid n-1,
		\qquad
		\ell\le n-1.
		\label{eq:ap-profile-proper-seed-obstruction}
	\end{equation}
	
	We now consider the minimal polynomial of
	$\beta$ over the intermediate field $K$.
	Since $d\mid t$, we have
	$K\subseteq\mathbb F_{q^t}$.
	The $q^e$-Frobenius on $\mathbb F_{q^t}$
	has order $t/d=\ell$ and fixed field
	$\mathbb F_{q^d}=K$. It therefore generates
	$\operatorname{Gal}(\mathbb F_{q^t}/K)$.
	
	Moreover, since $\beta$ has degree $t$
	over $\mathbb F_q$, its $q^e$-Frobenius
	orbit has exactly $\ell$ distinct elements.
	Consequently, its minimal polynomial over
	$K$ is
	\[
	M_\beta(X)=
	\prod_{j=0}^{\ell-1}
	(X-\beta^{q^{je}})
	\in K[X],
	\]
	with
	\[
	\deg M_\beta=\ell\le n-1.
	\]
	
	Every root of $g$ belongs to this
	$q^e$-Frobenius orbit. Since the roots
	of $g$ are distinct, we have
	\[
	g\mid M_\beta
	\qquad\text{in }E[X].
	\]
	Furthermore,
	\[
	\deg(M_\beta/g)
	=\ell-r<n-r=k.
	\]
	Thus $\operatorname{coeff}_n(M_\beta)$
	is a nonzero codeword of $C_g(n)$.
	
	The crucial point is that all coefficients
	of $M_\beta$ lie in
	$K=\mathbb F_{q^d}$, which has dimension
	$d$ over $\mathbb F_q$. Therefore,
	\[
	\operatorname{wt}_q
	\bigl(\operatorname{coeff}_n(M_\beta)\bigr)
	\le d\le e\le r.
	\]
	This contradicts the MRD minimum distance
	$r+1$.
	
	The argument applies to every possible
	seed degree $t\mid m_0$, including seeds
	lying in proper subfields of $E$.
	Hence the chosen profile cannot be realized
	by any polynomial-multiple MRD code over
	$\mathbb F_{q^{m_0}}$.
	
	By the definition of the every-extension
	threshold,
	\[
	M_{\mathrm{univ}}(q,n,r)
	>m_0=e(n-1),
	\]
	which proves
	\eqref{eq:profile-uniform-lower-bound}.
	Taking $e=2$, permitted when $r\ge2$
	and $k\ge r+2$, gives the bound $2n-1$.
	
	It remains to compare the lower and upper
	bounds along the specified parameter sequence.
	For every integer $s\ge2$, take
	\[
	e=r=s,\qquad n=s^2+2.
	\]
	The obstructed loss sequence consists of
	$s-1$ copies of $s$, followed by two ones.
	Since
	\[
	m_0=s(n-1)=s^3+s,
	\]
	the lower bound becomes
	\[
	M_{\mathrm{univ}}(q,s^2+2,s)
	\ge s^3+s+1.
	\]
	
	For the upper bound, we have
	\[
	\left\lceil\frac{k}{r}\right\rceil=s,
	\qquad
	rk+n-s=s^3+s+2,
	\qquad
	n-1=s^2+1.
	\]
	Applying
	Corollary~\ref{cor:polynomial-profile-bound}
	therefore gives
	\eqref{eq:profile-uniform-near-sharp-sequence}.
	
	The resulting upper and lower bounds
	differ by
	\[
	1+\left\lceil
	\log_q\left(\frac{8(s^2+1)}{q-1}\right)
	\right\rceil
	=O_q(\log n).
	\]
	Finally, since $n=s^2+2$, the lower
	bound grows as $n^{3/2}$ along this
	sequence. Hence no uniform
	$O_q(n)$ every-extension threshold
	is possible within the
	polynomial-multiple family.
\end{proof}

\begin{example}
	\label{ex:forbidden-extension}
	Take $r=2$, $n=6$, and $e=2$, so that
	$k=4$.
	The obstructed loss sequence is
	$(2,1,1)$, corresponding to the
	intersection profile $(4,2,1,0)$.
	
	We claim that this profile cannot be
	realized by a polynomial-multiple MRD
	code over $\mathbb F_{q^{10}}$.
	
	Indeed, Lemma~\ref{lem:ap-profile-all-points}
	would force any such generator to have
	the form
	\[
	g(X)=(X-\beta)(X-\beta^{q^2}),
	\]
	where the degree $t$ of $\beta$ over
	$\mathbb F_q$ satisfies $t\mid10$ and
	$t\ge6$. Hence $t=10$.
	
	The minimal polynomial of $\beta$
	over $\mathbb F_{q^2}$ is
	\[
	M_\beta(X)=
	\prod_{j=0}^{4}(X-\beta^{q^{2j}})
	\in\mathbb F_{q^2}[X].
	\]
	It has degree $5<6$ and is divisible
	by $g$. Therefore its coefficient
	vector is a nonzero codeword of
	$C_g(6)$ with rank weight at most $2$,
	contradicting the required MRD
	distance $3$.
	
	Thus
	\[
	M_{\mathrm{univ}}(q,6,2)\ge11.
	\]
\end{example}

This obstruction concerns realization
over every extension degree beyond a
uniform threshold. It does not rule out
realizing a prescribed profile over some
suitably chosen extension of degree
$O_q(n)$.

\section{Gabidulin recognition and non-Gabidulin realizations}
\label{sec:classification}

We now compare polynomial-multiple MRD codes with
generalized Gabidulin codes.
Let $1\le s<m$ satisfy $\gcd(s,m)=1$,
so that $\sigma_s(x)=x^{q^s}$ generates
$\operatorname{Gal}(E/\mathbb F_q)$.
For $a=(a_0,\ldots,a_{n-1})\in E^n$ with
$\mathbb F_q$-linearly independent coordinates,
the generalized Gabidulin code of dimension $k$ is
\begin{equation}
	\operatorname{Gab}_{k,s}(a)
	=
	\langle a,\sigma_s(a),\ldots,
	\sigma_s^{k-1}(a)\rangle_E.
	\label{eq:gab-definition}
\end{equation}
By the generalized Moore determinant criterion,
this is an MRD code with minimum rank distance
$n-k+1$
\cite{Gabidulin1985,KshevetskiyGabidulin2005,HTM}.

The class of generalized Gabidulin codes with a
fixed step $s$ is preserved under the semilinear
rank equivalences of
Section~\ref{sec:rank-equivalence}.
Indeed, for any field automorphism $\tau$ of $E$
and any $B\in\GL_n(q)$,
\[
\tau(\operatorname{Gab}_{k,s}(a))B
=
\operatorname{Gab}_{k,s}(\tau(a)B),
\]
since $\tau$ commutes with $\sigma_s$ and
$\sigma_s(B)=B$. Multiplication of an
$E$-linear code by a nonzero scalar does not
change the code.
Consequently, to exclude equivalence to
generalized Gabidulin codes, it suffices to
exclude all steps $s$ coprime to $m$.

We recall the following intersection criterion of
Horlemann-Trautmann and Marshall
\cite[Theorem~4.8]{HTM}.
Its nontrivial implication is taken from their work;
the converse follows from the generalized Moore property.

\begin{theorem}[Horlemann-Trautmann--Marshall criterion]
	\label{thm:HTM}
	Let $C\le E^n$ be an MRD code of dimension
	$2\le k<n\le m$, and let $1\le s<m$ satisfy $\gcd(s,m)=1$. Then $C$ is a generalized
	Gabidulin code with step $s$ if and only if
	\begin{equation}
		\dim_E(C\cap\sigma_s(C))=k-1.
		\label{eq:HTM}
	\end{equation}
\end{theorem}

\begin{proof}
The implication from the intersection condition to the generalized
Gabidulin form is \cite[Theorem~4.8]{HTM}.
Conversely, if $C=\operatorname{Gab}_{k,s}(a)$, then
\[
C+\sigma_s(C)
=\langle a,\sigma_s(a),\ldots,\sigma_s^k(a)\rangle_E
\]
has dimension $k+1$ by the generalized Moore criterion, since
$k+1\le n$. Hence
$\dim_E(C\cap\sigma_s(C))=k-1$.
\end{proof}

We now apply the Gabidulin intersection criterion
to polynomial-multiple codes. The factor-orbit
decomposition in Section~\ref{sec:global-hierarchy}
allows us to characterize their Gabidulin
members directly in terms of the roots of $g$,
without assuming that $g$ is squarefree or
splits over $E$.

\begin{theorem}
	\label{thm:root-classification}
	Let $g\in E[X]$ be monic of degree $r$, where
	$1\le r<n\le m$ and $k=n-r\ge2$.
	For every integer $s$ with $1\le s<m$ and $\gcd(s,m)=1$,
	the following are equivalent:
	\begin{enumerate}
		\item $C_g(n)$ is a generalized Gabidulin code
		with step $s$.
		
		\item $C_g(n)$ is MRD and
		\[
		\deg\gcd(g,\sigma_s(g))=r-1.
		\]
		
		\item There exists $\alpha\in E$ with
		$[\mathbb F_q(\alpha):\mathbb F_q]\ge n$
		such that
		\begin{equation}
			g(X)=\prod_{i=0}^{r-1}
			(X-\sigma_s^i(\alpha)).
			\label{eq:root-chain}
		\end{equation}
	\end{enumerate}
\end{theorem}

\begin{proof}
	Write $\sigma=\sigma_s$.
	
	\emph{Equivalence of (1) and (2).}
	Put $d=\deg\gcd(g,\sigma(g))$.
	Lemma~\ref{lem:intersection} gives
	\[
	\dim_E(C_g(n)\cap\sigma(C_g(n)))
	=\max\{0,k-r+d\}.
	\]
	Since $k-1>0$, this dimension equals $k-1$
	if and only if $d=r-1$.
	Theorem~\ref{thm:HTM} therefore proves
	the equivalence.
	
	\emph{From (2) to (3).}
	By Corollary~\ref{cor:canonical-multichains},
	the MRD property excludes the Frobenius-fixed
	part of $g$, and its canonical chains satisfy
	\[
	\sum_\lambda\deg P_\lambda
	=r-\deg\gcd(g,\sigma(g))=1.
	\]
	Since every summand is a positive integer,
	there is exactly one chain, starting with
	a linear factor $X-\alpha$. Comparing degrees,
	we obtain
	\[
	g(X)=\prod_{i=0}^{r-1}
	(X-\sigma^i(\alpha)).
	\]
	The factors are distinct because the chain
	is shorter than the full $\sigma$-orbit.
	
	Since $\sigma$ generates
	$\operatorname{Gal}(E/\mathbb F_q)$,
	all these roots are $\mathbb F_q$-conjugates
	of $\alpha$. Hence $g$ divides the minimal
	polynomial $M_\alpha\in\mathbb F_q[X]$.
	If
	\[
	t=[\mathbb F_q(\alpha):\mathbb F_q]<n,
	\]
	then
	\[
	\deg(M_\alpha/g)=t-r<k.
	\]
	Thus $\coeff_n(M_\alpha)$ is a nonzero
	codeword of $C_g(n)$. Since
	$M_\alpha\in\mathbb F_q[X]$, this codeword
	has rank weight $1$, contradicting the
	MRD minimum distance $r+1$.
	Therefore $t\ge n$.
	
	\emph{From (3) to (2).}
	Let
	\[
	t=[\mathbb F_q(\alpha):\mathbb F_q]\ge n,
	\qquad
	v=(1,\alpha,\ldots,\alpha^{n-1}).
	\]
	The coordinates of $v$ are linearly
	independent over $\mathbb F_q$.
	Since $t\mid m$ and $\gcd(s,m)=1$,
	the $\sigma$-orbit of $\alpha$ also has
	length $t$. Hence the roots in
	\eqref{eq:root-chain} are distinct.
	
	Their evaluation matrix
	\[
	H=
	\begin{pmatrix}
		v\\
		\sigma(v)\\
		\vdots\\
		\sigma^{r-1}(v)
	\end{pmatrix}
	\]
	has rank $r$, since its first $r$
	columns form a nonsingular Vandermonde
	matrix. Thus $H$ is a parity-check
	matrix for $C_g(n)$.
	
	For every rank-$r$ matrix
	$A\in\mathbb F_q^{n\times r}$,
	the coordinates of $vA$ remain
	$\mathbb F_q$-linearly independent.
	Moreover, since $\sigma(A)=A$,
	\[
	HA=
	\begin{pmatrix}
		vA\\
		\sigma(vA)\\
		\vdots\\
		\sigma^{r-1}(vA)
	\end{pmatrix}.
	\]
	This is a nonsingular generalized Moore
	matrix. Proposition~\ref{prop:mrd-criterion}
	therefore proves that $C_g(n)$ is MRD.
	
	Finally, the root sets of $g$ and
	$\sigma(g)$ have exactly $r-1$ common
	elements, since their $\sigma$-orbit
	has length $t\ge n>r$.
	Consequently,
	\[
	\deg\gcd(g,\sigma(g))=r-1,
	\]
	which proves (2).
\end{proof}

\begin{remark}
	\label{cor:recognition}
	For $k\ge2$, Theorem~\ref{thm:root-classification}
	shows that an MRD polynomial-multiple code
	$C_g(n)$ is generalized Gabidulin if and only if
	\[
	\deg\gcd(g,\sigma_s(g))=r-1
	\]
	for some integer $s$ with $1\le s<m$ and $\gcd(s,m)=1$.
	This is a recognition criterion within the
	MRD family, not an independent test for the
	MRD property.
	
	The restriction $k\ge2$ is essential.
	Every one-dimensional MRD code is generalized
	Gabidulin, although its polynomial generator
	need not split over $E$.
	At the other boundary, $r=1$, we have
	$g=X-\alpha$, and $C_g(n)$ is MRD precisely when
	\[
	[\mathbb F_q(\alpha):\mathbb F_q]\ge n.
	\]
	See also~\cite[Theorem~5.1]{HTM}.
\end{remark}

We use the following elementary observation to
lift arithmetic progressions of Frobenius exponents
modulo $m$ to ordinary arithmetic progressions.

\begin{lemma}
	\label{lem:lift-modular-progression}
	Let $T\subseteq\{0,\ldots,L\}$ contain at least
	two elements, and suppose that $m>2L$.
	If the elements of $T$ can be arranged as an
	arithmetic progression modulo $m$, then $T$
	is an ordinary arithmetic progression in
	the integers.
	
	Moreover, if $e>0$ is its ordinary common
	difference, then the modular step is congruent
	to $e$ or $-e$ modulo $m$.
\end{lemma}

\begin{proof}
	Let $t_1,\ldots,t_v$ be an ordering of $T$
	such that
	\[
	t_{i+1}-t_i\equiv s\pmod m
	\qquad(1\le i<v).
	\]
	Since $0\le t_i\le L$, all consecutive
	differences belong to $[-L,L]$.
	As $m>2L$, two integers in this interval
	that are congruent modulo $m$ must be equal.
	Hence
	\[
	t_{i+1}-t_i=d
	\qquad(1\le i<v)
	\]
	for some integer $d\ne0$.
	
	Thus $t_i=t_1+(i-1)d$, so the ordering is
	strictly increasing or decreasing.
	Consequently, $T$ is an ordinary arithmetic
	progression with positive step $e=|d|$.
	Since $d\equiv s\pmod m$, we also have
	$s\equiv\pm e\pmod m$, as claimed.
\end{proof}

\begin{corollary}
	\label{cor:non-gab-universal-profiles}
	Every admissible profile with $\delta_1\ge2$
	is realized by MRD polynomial-multiple codes
	that are not semilinearly rank equivalent
	to any generalized Gabidulin code over
	infinitely many extensions of $\mathbb F_q$.
\end{corollary}

\begin{proof}
	The assumption $\delta_1\ge2$ implies $r,k\ge2$.
	Fix the exponent set
	$T=\{0=e_1<\cdots<e_r=L\}$ constructed in
	Theorem~\ref{thm:universal-profile-realization}.
	For every $m>\max\{B_T,2L\}$ and every
	$\alpha\in\mathbb F_{q^m}$ of degree $m$,
	the polynomial $g_{T,\alpha}$ generates an
	MRD code with the prescribed profile.
	
	Suppose that this code is generalized Gabidulin
	with some step $s$ satisfying $\gcd(s,m)=1$.
	By Theorem~\ref{thm:root-classification},
	its roots form a consecutive $q^s$-Frobenius
	chain. Since all roots belong to the
	$q$-Frobenius orbit of $\alpha$, the first
	root is $\alpha^{q^c}$ for some integer $c$.
	As $\alpha$ has degree $m$, equality of the
	root sets implies
	\[
	T=\{c,c+s,\ldots,c+(r-1)s\}
	\quad\text{in }\mathbb Z/m\mathbb Z.
	\]
	
	We distinguish two cases.
	
	If $T$ is not an ordinary arithmetic
	progression, Lemma~\ref{lem:lift-modular-progression}
	excludes such a modular progression whenever
	$m>2L$. Thus the resulting codes are not
	generalized Gabidulin for all sufficiently
	large $m$.
	
	Now suppose that $T$ is an ordinary arithmetic
	progression with positive step $e$.
	The construction in
	Theorem~\ref{thm:universal-profile-realization}
	has $\delta_1-1\ge1$ internal gaps, so
	$e\ge2$. Choose arbitrarily large multiples
	$m$ of $e$ satisfying $m>\max\{B_T,2L\}$.
	By Lemma~\ref{lem:lift-modular-progression},
	any possible modular step must satisfy
	$s\equiv\pm e\pmod m$. Hence
	\[
	\gcd(s,m)=\gcd(e,m)=e\ge2,
	\]
	contrary to the requirement $\gcd(s,m)=1$.
	
	Therefore, in either case, the prescribed
	profile has non-generalized-Gabidulin MRD
	realizations over infinitely many extension
	degrees. Since the class of generalized
	Gabidulin codes is preserved under semilinear
	rank equivalence, these codes are not
	equivalent to any generalized Gabidulin code.
\end{proof}

For the pattern $T=\{0,2,5\}$ of Example~\ref{ex:profile-construction},
every successful parameter of degree $m>10$ therefore gives a
non-generalized-Gabidulin code. In particular, the successful degree-$28$
parameters in Example~\ref{ex:binary-profile-budget} yield such
$[9,6,4]$ MRD codes over $\mathbb F_{2^{28}}$.

\paragraph{An extension-degree contrast.}
Take $q=2$, $n=6$, $r=2$, $k=4$, and $T=\{0,2\}$.
For $\alpha\in\mathbb F_{2^m}$ of degree $m\ge6$,
the polynomial $g_{T,\alpha}=(X-\alpha)(X-\alpha^4)$ has
intersection profile $(4,2,1,0)$, independently of the MRD
property. Here $\Delta_T=(k+1)(2^2-2)=10$, giving three
contrasting extension degrees:
\begin{itemize}[leftmargin=*,topsep=2pt,itemsep=1pt,parsep=0pt]
\item \emph{$m=10$:} No polynomial-multiple MRD code realizes
this profile, by Example~\ref{ex:forbidden-extension}.
\item \emph{$m=11$:} Every full-degree $\alpha$ gives an MRD code
by Theorem~\ref{thm:pattern-proportion}; it is generalized
Gabidulin with step $2$ by Theorem~\ref{thm:root-classification}.
\item \emph{$m=12$:} Every full-degree $\alpha$ again gives an
MRD code, but none is equivalent to a generalized Gabidulin
code. Indeed, Lemma~\ref{lem:lift-modular-progression} forces
any possible modular step to be $\pm2\pmod{12}$, neither of
which is coprime to $12$.
\end{itemize}
Thus the same intersection profile is shared by generalized
Gabidulin and non-generalized-Gabidulin MRD codes.

\begin{corollary}
	\label{cor:profile-count}
	Let $p_{\le r}(k)$ denote the number of
	partitions of $k$ whose largest part
	is at most $r=n-k$.
	
	Over every extension covered by
	Theorem~\ref{thm:main-realization},
	exactly $p_{\le r}(k)$ Frobenius
	intersection profiles occur among
	$[n,k]$ MRD polynomial-multiple codes.
	Consequently, there are at least
	$p_{\le r}(k)$ distinct semilinear
	rank-equivalence classes.
	
	If $r,k\ge2$, then over infinitely
	many extension degrees, all
	$p_{\le r}(k)-1$ profiles with
	$\delta_1\ge2$ can be realized
	simultaneously by pairwise inequivalent
	non-generalized-Gabidulin MRD codes.
\end{corollary}

\begin{proof}
	By Proposition~\ref{prop:necessary-profile},
	the positive loss sequences of MRD codes
	are partitions of $k$ whose largest part
	is at most $r$.
	Theorem~\ref{thm:main-realization}
	realizes every such partition over each
	extension covered by its bound.
	Thus exactly $p_{\le r}(k)$ profiles occur.
	Since Frobenius intersection profiles are
	invariant under semilinear rank equivalence
	(Section~\ref{sec:rank-equivalence}),
	distinct profiles yield inequivalent codes.
	
	Now suppose that $r,k\ge2$.
	Exactly one admissible loss sequence has
	$\delta_1=1$, namely $(1,\ldots,1)$.
	For each of the remaining
	$p_{\le r}(k)-1$ profiles, fix an exponent
	set $T$ constructed in
	Theorem~\ref{thm:universal-profile-realization}.
	Since there are only finitely many such
	profiles, only finitely many exponent sets
	need to be considered.
	
	Let $M$ be the least common multiple of
	the positive steps $e\ge2$ of those
	sets $T$ that are ordinary arithmetic
	progressions, with $M=1$ if there are none.
	The construction ensures that
	$\delta_1\ge2$ gives at least one internal
	gap, so every such progression has step
	$e\ge2$.
	
	Choose any sufficiently large multiple
	$m$ of $M$ such that
	\[
	m>\max_T\{B_T,2\max T\}.
	\]
	By Lemma~\ref{lem:fixed-pattern-mrd},
	every degree-$m$ parameter
	$\alpha\in\mathbb F_{q^m}$ yields an MRD
	code with the prescribed profile for
	each chosen exponent set $T$.
	
	We show that none of these codes is
	generalized Gabidulin.
	If $T$ is not an ordinary arithmetic
	progression, this follows from
	Lemma~\ref{lem:lift-modular-progression},
	since $m>2\max T$.
	Otherwise, $T$ has an ordinary step
	$e\ge2$ dividing $m$.
	The same lemma shows that any possible
	modular progression step $s$ satisfies
	$s\equiv\pm e\pmod m$, and hence
	\[
	\gcd(s,m)=e\ge2.
	\]
	Thus no admissible generalized Gabidulin
	step exists.
	
	Since the generalized Gabidulin class is preserved under
	semilinear rank equivalence, the resulting codes are not
	semilinearly equivalent to generalized Gabidulin codes.
	Moreover, they are pairwise inequivalent because their
	Frobenius intersection profiles are distinct.
	
	Arbitrarily large multiples of $M$ satisfy
	the required bound, proving the simultaneous
	realization over infinitely many extensions.
\end{proof}

\begin{example}
\label{ex:same-profile-inequivalent}
Let $n=12$, $r=4$, and $k=8$, and take
\[
T_1=\{0,2,5,9\},\qquad T_2=\{0,2,6,9\}.
\]
For every integer $m>\max\{B_{T_1},B_{T_2},18\}$, choose
$\alpha\in\mathbb{F}_{q^m}$ of degree $m$ over
$\mathbb{F}_q$ and set
\[
g_i(X)=\prod_{a\in T_i}(X-\alpha^{q^a}),
\qquad C_i=C_{g_i}(12)\quad(i=1,2).
\]
Both codes are MRD by Lemma~\ref{lem:fixed-pattern-mrd}.
The internal gaps of $T_1$ and $T_2$ have lengths
$(1,2,3)$ and $(1,3,2)$, respectively, and both external
gaps have length $m-10$. Hence Theorem~\ref{thm:orbit-gap-formula}
gives, for every $j\ge0$,
\[
\begin{aligned}
D_j(g_1)=D_j(g_2)
={}&4+\min\{j,1\}+\min\{j,2\}\\
 &+\min\{j,3\}+\min\{j,m-10\}.
\end{aligned}
\]
In particular, both codes have the consecutive Frobenius
intersection profile $(8,4,1,0,\ldots)$.
Here $D_3(g_1)=D_3(g_2)=13>n=12$, so
$d_3(C_1)=d_3(C_2)=\max\{0,12-13\}=0$.
Thus the zero at $j=3$ is obtained by truncation.

Nevertheless, $T_1\cap(T_1+5)=\{5\}$ whereas
$T_2\cap(T_2+5)=\varnothing$, with shifts taken modulo $m$.
Since $\sigma^5(g_i)$ has root-exponent set $T_i+5$,
Lemma~\ref{lem:intersection} yields
\[
\begin{aligned}
\dim_E(C_1\cap\sigma^5 C_1)&=4+1=5,\\
\dim_E(C_2\cap\sigma^5 C_2)&=4+0=4.
\end{aligned}
\]
These dimensions are invariant under semilinear rank
equivalence: if $C'=\tau(C)A$, where $\tau$ is a field
automorphism and $A\in\mathrm{GL}_n(q)$, then
\[
C'\cap\sigma^5(C')=\tau(C\cap\sigma^5(C))A.
\]
Thus $C_1$ and $C_2$ are not semilinearly rank equivalent.
Moreover, neither exponent set is an ordinary arithmetic
progression. By Lemma~\ref{lem:lift-modular-progression}
and Theorem~\ref{thm:root-classification}, both codes are
non-generalized-Gabidulin under the stated bound on $m$.

Thus, even within the polynomial-multiple family,
the complete degree hierarchy $(D_j)$ does not determine
the semilinear rank-equivalence class. Since
$p_{\le4}(8)=15$, combining this pair with the
other admissible profiles yields at least $16$ distinct
semilinear rank-equivalence classes of $[12,8]$
polynomial-multiple MRD codes over every sufficiently
large extension.
\end{example}

\begin{remark}
	The number $p_{\le r}(k)$ counts
	Frobenius intersection profiles exactly,
	but gives only a lower bound on the
	number of semilinear rank-equivalence
	classes. The preceding example shows that this
	lower bound can be strict.
	
	For arbitrary admissible profiles with $\delta_1\ge2$,
	the non-generalized-Gabidulin realizations above are
	guaranteed over infinitely many extension degrees. The
	argument does not claim realization over every sufficiently
	large extension degree.
	
	Finally, the condition $\delta_1\ge2$
	for ordinary $q$-Frobenius alone does
	not rule out generalized Gabidulin
	codes with other admissible steps.
	The argument above excludes all steps
	$s$ coprime to $m$.
\end{remark}

\section{Conclusion}
\label{sec:conclusion}

We established that every numerically admissible Frobenius intersection profile is realized by a polynomial-multiple MRD code over every sufficiently large finite extension. We derived a factor-orbit gap formula and obtained effective upper and lower bounds for the every-extension threshold, which differ by only $O_q(\log n)$ along an infinite family of parameters. We also characterized the generalized Gabidulin members of this
family and established non-generalized-Gabidulin realizations of
every profile with initial loss at least two over infinitely many
extension degrees. Finally, we exhibited semilinearly rank-inequivalent MRD codes
whose polynomial generators have identical complete degree
hierarchies $(D_j)_{j\ge0}$.

\section*{Funding}
This work was supported by the National Natural Science Foundation of China (grant numbers 12571003 and 12501006) and the Basic and Applied Basic Research Foundation of Guangdong Province (grant number 2024A1515010589).

\section*{Data Availability}
No datasets were generated or analysed during the current study.

\section*{Competing Interests}
The authors declare no competing interests.

\section*{AI Use Declaration}
During the preparation and revision of this manuscript, the authors used
generative AI tools for language editing and polishing, literature
searches, comparison with related work for assessing the novelty and
positioning of the manuscript, and the construction of examples presented
in the paper. The research questions, mathematical motivation, proofs of
the results, and the overall structure and organization of the manuscript
were developed by the authors. All AI-assisted references, comparisons,
and examples were independently checked and verified by the authors. The
authors take full responsibility for the accuracy and content of the
manuscript.


\begingroup
\small
\begin{thebibliography}{99}

\bibitem{NPH2020}A. Neri, S. Puchinger and A.-L. Horlemann-Trautmann,
Equivalence and characterizations of linear rank-metric codes based on invariants,
Linear Algebra and its Applications \textbf{603} (2020), 418--469.
\href{https://doi.org/10.1016/j.laa.2020.06.014}{doi:10.1016/j.laa.2020.06.014}.

\bibitem{Delsarte1978}P. Delsarte,
Bilinear forms over a finite field, with applications to coding theory,
Journal of Combinatorial Theory, Series A \textbf{25} (1978), 226--241.
\href{https://doi.org/10.1016/0097-3165(78)90015-8}{doi:10.1016/0097-3165(78)90015-8}.

\bibitem{Gabidulin1985}E.~M. Gabidulin,
Theory of codes with maximum rank distance,
Problems of Information Transmission \textbf{21} (1985), 1--12.

\bibitem{KshevetskiyGabidulin2005}A. Kshevetskiy and E. Gabidulin,
The new construction of rank codes,
in: Proceedings of the 2005 IEEE International Symposium on Information
Theory, IEEE, 2005, pp.~2105--2108.
\href{https://doi.org/10.1109/ISIT.2005.1523717}{doi:10.1109/ISIT.2005.1523717}.

\bibitem{BartoliZhou2020}D. Bartoli and Y. Zhou,
Asymptotics of Moore exponent sets,
Journal of Combinatorial Theory, Series A \textbf{175} (2020), 105281.
\href{https://doi.org/10.1016/j.jcta.2020.105281}{doi:10.1016/j.jcta.2020.105281}.

\bibitem{NeriEtAl2018}A. Neri, A.-L. Horlemann-Trautmann,
T. Randrianarisoa and J. Rosenthal,
On the genericity of maximum rank distance and Gabidulin codes,
Designs, Codes and Cryptography \textbf{86} (2018), 341--363.
\href{https://doi.org/10.1007/s10623-017-0354-4}{doi:10.1007/s10623-017-0354-4}.

\bibitem{Sheekey2016}J. Sheekey,
A new family of linear maximum rank distance codes,
Advances in Mathematics of Communications \textbf{10} (2016), 475--488.
\href{https://doi.org/10.3934/amc.2016019}{doi:10.3934/amc.2016019}.

\bibitem{HTM}A.-L. Horlemann-Trautmann and K. Marshall,
New criteria for MRD and Gabidulin codes and some rank-metric code
constructions,
Advances in Mathematics of Communications \textbf{11} (2017), 533--548.
\href{https://doi.org/10.3934/amc.2017042}{doi:10.3934/amc.2017042}.

\end{thebibliography}
\endgroup
\end{document}